\documentclass[oneside]{amsart}
\usepackage{etex}
\usepackage[T1]{fontenc}
\usepackage{lmodern}
\usepackage[utf8]{inputenc}
\usepackage[a4paper]{geometry}
\usepackage{amssymb,amsfonts,amsmath}
\usepackage{comment}
\usepackage[all,arc]{xy}
\usepackage{enumerate}
\usepackage{mathrsfs,mathtools}
\usepackage{todonotes,booktabs}
\usepackage{stmaryrd}
\usepackage[shortlabels]{enumitem}

\usepackage{tikz}
\usepackage{tikz-cd}
\usetikzlibrary{trees}
\usetikzlibrary[shapes]
\usetikzlibrary[arrows]
\usetikzlibrary{patterns}
\usetikzlibrary{fadings}
\usetikzlibrary{backgrounds}
\usetikzlibrary{decorations.pathreplacing}
\usetikzlibrary{decorations.pathmorphing}
\usetikzlibrary{positioning}
\def\biblio{\bibliography{duality}\bibliographystyle{alpha}}

\usepackage{xcolor} 
\usepackage{graphicx}

\usepackage[pagebackref]{hyperref}

\definecolor{dark-red}{rgb}{0.5,0.15,0.15}
\definecolor{dark-blue}{rgb}{0.15,0.15,0.6}
\definecolor{dark-green}{rgb}{0.15,0.6,0.15}
\hypersetup{
    colorlinks, linkcolor=dark-red,
    citecolor=dark-blue, urlcolor=dark-green
}
\newcommand{\ctimes}{\boxtimes}
\newcommand{\iHom}{\underline{\operatorname{Hom}}}

\newcommand{\noteD}[1]{{\color{dark-red}Drew: #1}}

\renewcommand*{\backref}[1]{}
\renewcommand*{\backrefalt}[4]{%
  \ifcase #1 %
No citations.
  \or
(cit. on p. #2).%
  \else
(cit on pp. #2).%
  \fi%
}
\usepackage{subfiles}
\newtheorem{thmx}{Theorem}
 
\newtheorem{corx}{Corollary}
 
\usepackage[nameinlink,capitalise,noabbrev]{cleveref}

\newtheorem{thm}{Theorem}[section]
\newtheorem{cor}[thm]{Corollary}
\newtheorem{prop}[thm]{Proposition}
\newtheorem{lem}[thm]{Lemma}

\theoremstyle{definition}
\newtheorem{defn}[thm]{Definition}

\newtheorem{exmp}[thm]{Example}
\newtheorem{example}[thm]{Example}

\theoremstyle{remark}
\newtheorem{rem}[thm]{Remark}

\makeatletter
\let\c@equation\c@thm
\makeatother
\numberwithin{equation}{section}

\DeclareMathOperator{\Sp}{Sp}
\DeclareMathOperator{\Hom}{Hom}

\DeclareMathOperator{\colim}{colim}
\DeclareMathOperator{\cA}{\mathcal{A}}
\DeclareMathOperator{\cC}{\mathcal{C}}
\DeclareMathOperator{\cD}{\mathcal{D}}
\DeclareMathOperator{\cE}{\mathcal{E}}

\DeclareMathOperator{\cS}{\mathcal{S}}
\DeclareMathOperator{\cI}{\mathcal{I}}

\DeclareMathOperator{\cG}{\mathcal{G}}

\DeclareMathOperator{\fX}{\mathfrak{X}}

\DeclareMathOperator{\Ext}{Ext}

\DeclareMathOperator{\Tor}{Tor}
\DeclareMathOperator{\Spec}{Spec}
\DeclareMathOperator{\Mod}{Mod}
\DeclareMathOperator{\Stable}{Stable}
\DeclareMathOperator{\Comod}{Comod}
\DeclareMathOperator{\Ch}{Ch}
\DeclareMathOperator{\D}{\mathcal{D}}

\DeclareMathOperator{\Loc}{Loc}

\DeclareMathOperator{\Thick}{Thick}
\DeclareMathOperator{\Ind}{Ind}

\DeclareMathOperator{\QCoh}{QCoh}
\DeclareMathOperator{\IndCoh}{IndCoh}

\newcommand{\N}{\mathbb{N}}
\newcommand{\Q}{\mathbb{Q}}

\newcommand{\sh}{H}

\newcommand{\Locid}[1]{\mathrm{Loc}_{#1}^{\otimes}}

\newcommand{\fZ}{\mathfrak{Z}}

\newcommand{\LM}{\widehat{\textnormal{Mod}}_A}
\newcommand{\wComod}{\widehat{\Comod}}

\newcommand{\wPsi}{{\widehat{\Psi}}}

\newcommand{\wA}{{\widehat{A}}}
\newcommand{\wpsi}{{\widehat{\psi}}}
\newcommand{\wI}{\widehat{I}}
\newcommand{\wM}{\widehat{M}}
\newcommand{\wN}{\widehat{N}}

\newcommand{\p}{\frak{p}}
\newcommand{\m}{\frak{m}}

\newcommand{\cal}{\mathcal}
\newcommand{\xr}{\xrightarrow}

\newcommand{\Z}{\mathbb{Z}}

\Crefname{figure}{Figure}{Figures}
\Crefname{assu}{Assumption}{Assumptions}
\Crefname{thmx}{Theorem}{Theorems}

\newcommand{\psilim}{\lim_{\Psi,k}}

\renewcommand{\frak}{\mathfrak}
\newcommand{\wDelta}{\widehat \Delta}
\newcommand{\ilim}{\lim{\vphantom \lim}}
\let\lim\relax
\DeclareMathOperator{\lim}{lim}
\newcommand{\Dtors}{\D^{{I-\mathrm{tors}}}}
\newcommand{\Dhtors}{\D^{{I-\mathrm{tors}}}_{\mathrm{cohom}}
}
\newcommand{\Dcmpl}{\D^{{I-\mathrm{cmpl}}}}

\newcommand{\cComod}{\Comod^c}

\newcommand{\cctimes}{\hat{\ctimes}}

\newcommand{\fm}{\mathfrak{m}}

\newcommand{\loc}{\mathrm{loc}}

\newcommand\noloc{%
  \nobreak
  \mspace{6mu plus 1mu}
  {:}
  \nonscript\mkern-\thinmuskip
  \mathpunct{}
  \mspace{2mu}
}

\title{Derived completion for comodules}

\author{Tobias Barthel}
\address{Department of Mathematical Sciences, University of Copenhagen, Universitetsparken 5, 2100 K{\o}benhavn {\O}, Denmark}
\email{tbarthel@math.ku.dk}

\author{Drew Heard}
\address{Fakult{\"a}t f{\"u}r Mathematik, Universit{\"a}t Regensburg}
\email{drew.k.heard@gmail.com}

\author{Gabriel Valenzuela}
\address{Max-Planck-Institut f\"ur Mathematik, Bonn, Germany}
\email{gvalenzuela@mpim-bonn.mpg.de}

\date{\today}


\begin{document}

\begin{abstract}
The objective of this paper is to introduce and study completions and local homology of comodules over Hopf algebroids, extending previous work of Greenlees and May in the discrete case. In particular, we relate module-theoretic to comodule-theoretic completion, construct various local homology spectral sequences, and derive a tilting-theoretic interpretation of local duality for modules. Our results translate to quasi-coherent sheaves over global quotient stacks and feed into a novel approach to the chromatic splitting conjecture. 
\end{abstract}

\maketitle

{\hypersetup{linkcolor=black}\tableofcontents}
\def\biblio{}

\section*{Introduction}
Completion of non-finitely generated modules is pervasive throughout stable homotopy theory, as amply demonstrated in \cite{GM_MU}, for example. Its left derived functors can be interpreted as a type of local homology, which in turn gives rise to a local duality theory for modules over commutative rings \cite{gm_localhomology}. On the other hand, the theory of comodules over a Hopf algebroid arises naturally in the context of generalized homology theories \cite[Appendix A1]{ravenel_86}; the homology of a spectrum has the structure of a comodule over the ring of cooperations. Moreover, in light of \cite{naumann_stack_2007}, we can translate results about comodules into quasi-coherent sheaves over certain algebraic stacks. The latter have shown to be fundamental in moduli problems where the objects one wishes to parameterize have nontrivial automorphisms. 

The goal of this paper is to generalize local homology from modules over commutative rings to comodules over Hopf algebroids, which among other applications plays a central role in an algebraic approach to Hopkins' chromatic splitting conjecture \cite{ctc}. Algebraically, this extends the work of Greenlees and May \cite{gm_localhomology} on derived functors of completion, and in geometric terms it is akin to the passage from affine schemes to quotient stacks. However, while local cohomology admits a canonical and well-behaved extension from modules over commutative rings to comodules over Hopf algebroids, the corresponding generalization of local homology is considerably more complicated.

This complication is already visible at the non-derived level: Unlike the case of modules, for a Hopf algebroid $(A,\Psi)$ the naive completion $C^I(-) = \lim_k(A/I^k\otimes -)$ at an ideal $I \subseteq A$ does not usually define an endofunctor on the category of comodules $\Comod_{\Psi}$, but rather takes values in a category of completed comodules \cite{dev_cor}. To remedy this, one has to replace the limit $\lim_k$ of the underlying $A$-modules by the inverse limit in comodules, which leads to a comodule completion functor $C_{\Psi}^I$. We thus begin in \cref{sec:completion} with an analysis of these non-derived completion functors and in particular the relation between $C^I$ and $C_{\Psi}^I$.

Given an inverse system of $\Psi$-comodules, the key problem thus becomes to compare the comodule limit with the underlying module limit, and our first result provides conditions under which the former can be computed from the latter. This motivates the introduction of a class of Hopf algebroids which we call true-level (with respect to the ideal $I$), see \Cref{def:truelevel}.  We then use a theorem of Enochs to deduce concrete conditions that imply the true-level property; a particular example of a true-level Hopf algebroid highly relevant for applications to stable homotopy theory is given by $(A,\Psi)= (\cE_*,\cE_*\cE)$ for a variant $\cE$ of Johnson--Wilson theory due to Baker~\cite{baker_$i_n$-local_2000}. Note that we always write $\ctimes$ for the underived tensor product.

Let $N$ be a complete $\Psi$-comodule. If $(A,\Psi)$ is a true-level Hopf algebroid with respect to $I$, then we prove that the $A$-module $\iota N$ defined by the following pullback square
\[
\xymatrix{\iota N \ar[r] \ar[d] & \Psi \ctimes N \ar[d] \\
N \ar[r] & C^I(\Psi \ctimes N)}
\]
acquires a natural structure as a $\Psi$-comodule. Here the natural map $\iota N \to N$ is injective and $\iota N$ is the largest possible $A$-submodule of $N$ that carries a natural $\Psi$-comodule structure. We note that $\iota$ roughly speaking plays the role of a (non-existent) right adjoint to $C^I$ on the category of comodules.

\begin{thmx}[\cref{thm:completionformula}]
If $(A,\Psi)$ is true-level, then for any $\Psi$-comodule $M$, there is an equivalence of $\Psi$-comodules $C^I_{\Psi}(M) \simeq \iota C^I(M)$.
\end{thmx}

We then move on to a study of derived completion. For a suitable Hopf algebroid $(A,\Psi)$ and ideal $I \subseteq A$ we construct a local homology functor for comodules $\Lambda^I$. Our construction is dictated by the general local duality framework of \cite{bhv}, and \Cref{sec:derivedcompletion} studies the properties of the resulting functors. In particular, we work with a suitable enlargement $\Stable_{\Psi}$ of the derived category of comodules with some desirable categorical properties; geometrically speaking, this corresponds to the passage from quasi-coherent to ind-coherent sheaves. 

One of the first new phenomena we encounter is that the local homology of a comodule can be non-zero both in positive and negative degrees, which may be interpreted as a measure of the stackiness of the Hopf algebroid under consideration. Consequently, the relation between derived functors of completion and local homology turns out to be more subtle.

For an arbitrary Hopf algebroid we construct a spectral sequence of the form
\[
E_2^{p,q} = \ilim^p_{\Psi,k}\Tor^{\Psi}_q(A/I^k,M) \implies H_{q-p}(\Lambda^I(M))
\]
computing the local homology of a comodule $M$ in terms of more familiar functors. In parallel to the equivalence of local homology with $I$-adic completion when restricted to finitely presented modules, if $A$ is Noetherian and $M$ is a compact comodule, we show that this spectral sequence collapses to yield
\[ 
\xymatrix{H_{-s}(\Lambda^{I}M) \ar[r]^-{\simeq} & {\ilim}_{\Psi,k}^{s}M \boxtimes A/I^k}
\]
for all $s\geq0$. Moreover, we give an example to show that, contrary to the case of $A$-modules, the functors $H_*(\Lambda^I(M))$ cannot, in general, be computed by the left or right derived functors of $C^I_{\Psi}(-)$.  

In the case of a discrete Hopf algebroid $(A,A)$ there is a natural equivalence $\Stable_{A} \simeq \cD(A)$. We can thus use the relationship between $\Lambda^I$ and the left derived functors of completion (known as the derived functors of $L$-completion) to prove some results about the derived functors of $C^I$ on $\Mod_A$. We also produce a criterion for when an $A$-module is $L$-complete, i.e., for when $M$ is in the category $\LM$ of $L$-complete $A$-modules, generalizing Bousfield and Kan's $\Ext$-$p$ completeness criterion.  

\begin{thmx}[$\Ext$-$I$ completeness criterion, \cref{thm:Lcompletecriterion}]
    Let $A$ be a commutative ring and $I\subseteq A$ an ideal generated by a regular sequence $x_1, x_2, \dots ,x_n$. If $M$ is an $A$-module, then
    \[
        M \text{ is $L$-complete } \Longleftrightarrow \Ext_A^q(x_i^{-1}A/(x_1,\dots, x_{i-1}),M)=0 \text{ for all }1 \le i \le n \text{ and all } q \ge 0.
    \]
\end{thmx}

In fact, this is a consequence of a more general result that characterizes those $M \in \Stable_{\Psi}$ for which $\Lambda^IM \simeq M$, see \Cref{cor:Lamdalocal}. 

In the final section, we turn to torsion and complete objects in derived categories of comodules. A priori, there are at least three different notions of what it could mean for an object $M \in \D(\Psi)$ to be torsion with respect to an ideal $I \subseteq A$: 
	\begin{enumerate}
		\item $M$ is in the smallest localizing ideal of $\D(\Psi)$ generated by $A/I$, denoted $\Dtors(\Psi)$. 
		\item $M$ is in the image of the canonical functor from the derived category of the abelian category of $I$-torsion  $\Psi$-comodules, denoted $\cD(\Comod_{\Psi}^{I-\mathrm{tors}})$.
		\item The homology groups $H_nM$ are $I$-torsion  $\Psi$-comodules for all $n \in \Z$. 
	\end{enumerate}
One gets analogous definitions for complete objects by replacing localizing with colocalizing and $I$-torsion with $I$-complete where appropriate. 

The goal of \Cref{sec:tilting} is to compare these notions and use this to prove a tilting-type equivalence between torsion and complete objects in $\D(\Psi)$. When working with comodules, the difficulties intrinsic to complete objects persist at the level of the derived category; while we can show that the three notions above coincide in the case of torsion objects, we can only conclude the same for complete objects when working over a discrete Hopf algebroid.

\begin{thmx}[\cref{thm:derder}]\label{thm:d}
Let $(A,\Psi)$ be an Adams Hopf algebroid and $I \subseteq A$ a finitely generated invariant ideal. 
	\begin{enumerate}
		\item Suppose $I$ is generated by a weakly proregular sequence. If $(A,\Psi)=(A,A)$ is discrete, then there is a canonical equivalence 
		between the right completion of $\cD^{-}(\LM)$ and $\Dcmpl(A)$. Moreover, an object $M \in \cD(A)$ is $I$-complete if and only if the homology groups $H_*M$ are $L$-complete. 
		\item \sloppy If $I$ is generated by a regular sequence, then there is a canonical equivalence $\cD(\Comod_{\Psi}^{I-\mathrm{tors}}) \simeq \Dtors(\Psi)$. Moreover, an object $M \in \D(\Psi)$ is $I$-torsion if and only if the homology groups $H_*M$ are $I$-torsion. 
	\end{enumerate}
\end{thmx}

There are a number of results in the literature closely related to \Cref{thm:d}. For example, in \cite[Cor.~3.32]{psy} it is proven that $M \in \cD(A)$ is cohomologically $I$-torsion (that is, the canonical morphism from $\mathbf{R}\Gamma_IM \to M$ is an equivalence, where $\mathbf{R}\Gamma_IM$ denotes the total derived functor of $I$-torsion of $M$) if and only if the homology groups $H_*M$ are $I$-torsion. Moreover, in unpublished work Rezk has constructed a version of the derived category of $L$-complete modules and has proven a version of the second part of \Cref{thm:d}(1), see \cite[Thm.~9.2]{rezk_analytic_2013}. 

Recall that the abelian categories of torsion and $L$-complete modules are not equivalent in general; for example, the former is Grothendieck while the latter is not. On the other hand, the subcategories $\Dcmpl(A)$ and $\Dtors(A)$ are known to be equivalent, see \cite[Thm.~3.11]{bhv} for example. Consequently, the latter together with the previous theorem allow us to deduce the following tilting-theoretic interpretation of local duality for commutative rings. 

\begin{corx}
For any commutative ring $A$ and $I \subseteq A$ a finitely generated ideal, local homology and local cohomology induce mutual inverse symmetric monoidal equivalences
\[
\xymatrix{\Lambda^I\colon \cD(\Mod_{A}^{I-\mathrm{tors}}) \ar@<0.5ex>[r]^-{\sim} & \cD(\LM)\noloc \Gamma_I \ar@<0.5ex>[l]^-{\sim}}
\]
where $ \cD(\LM)$ denotes the right completion of $\cD^{-}(\LM)$. 
\end{corx}

Once again, in the case of a general Hopf algebroid we were unable to obtain such a result. Indeed, there seems to be no good candidate for a derived category fitting the right hand side in the equivalence above.

Finally, let us say a few words about the geometric interpretation of this work. The equivalence between the categories of commutative rings and affine schemes extends to an equivalence between the categories of Hopf algebroids and certain algebraic stacks \cite{naumann_stack_2007}. Prominent examples of these stacks in homotopy theory are the moduli stack of $1$-dimensional formal groups and its Lubin--Tate substacks. Let $\fX$ be the stack presented by an Adams Hopf algebroid $(A,\Psi)$. We will write $\QCoh(\fX)$ for the category of quasi-coherent sheaves over $\fX$. Let $\cI$ be the ideal sheaf corresponding to an invariant ideal $I\subseteq A$. This determines a closed substack $\fZ$ of $\fX$, along which we can consider both completion and torsion functors on $\QCoh({\fX})$ together with their corresponding derived functors. In light of the symmetric monoidal equivalence between $\Comod_\Psi$ and $\QCoh({\fX})$ \cite[Prop. 5.37]{bhv}, and the definition of $\fZ$ in terms of $I$, we see that the corresponding subcategories of complete objects, resp. torsion objects, are equivalent as well.

On the other hand, the category of quasi-coherent sheaves over $\fX$ is a Grothendieck abelian category, and so we can consider its derived category $\D(\fX)=\D(\QCoh(\fX))$. Following \Cref{defn:Dtors}, there is a subcategory of $\D(\fX)$ of $\cI$-torsion objects, as well as a subcategory of $\cI$-complete objects in $\D(\Spec(A))$ in the case of a discrete Hopf algebroid. The equivalence between $\Comod_\Psi$ and $\QCoh({\fX})$ yields a symmetric monoidal equivalence between their corresponding derived categories that restricts to equivalences of subcategories of complete objects, resp. torsion objects, at the derived level. We can therefore translate all the results in this paper to the context of stacks and their abelian and derived categories of quasi-coherent sheaves. Although for the sake of conciseness we keep the exposition in the language of comodules over Hopf algebroids, our results should be of independent interest in the aforementioned geometric setting. We end this section with a dictionary between the algebraic and geometric contexts.

\begin{table}[h]
	\begin{tabular}{|l|l|}
		\hline
		Commutative algebra & Algebraic geometry \\ \hline
	 Hopf algebroid map $(A,A)\to(A,\Psi)$ & Stack presentation $\fX\to \Spec(A)$ \\ 
	 Invariant ideal $I\subseteq A$& Closed substack $\fZ\subseteq \fX$ \\ 
	 $\Comod_\Psi$, $\D(\Psi)$ & $\QCoh(\fX)$, $\D(\fX)$ \\ 
	 $\Stable_\Psi=\Ind(\Thick(\cG_\Psi))$ & $\IndCoh_{\fX}=\Ind(\Thick(\cG_{\fX}))$  \\ \hline
	\end{tabular}
	\end{table}

\subsection*{Conventions}

	We always assume that our Hopf algebroids are flat. Moreover, we will write $\boxtimes$ for the underived tensor product of comodules and $\otimes$ for the derived tensor product. We will denote the internal hom-object in a category by $\iHom$. For a cocomplete category $\cC$, we let $\cC^{\omega}$ denote the full subcategory of compact objects in $\cC$.

	We work with $\infty$-categories throughout this document, specifically the quasi-categories of Lurie and Joyal \cite{htt,ha}. Unless otherwise noted, all functors between stable $\infty$-categories are assumed to be exact and all subcategories of stable $\infty$-categories are assumed to be stable subcategories. We follow the convention of \cite{htt} and say that a functor between presentable stable $\infty$-categories is continuous if it preserves filtered colimits.  Given a collection of objects $\cS$ in an $\infty$-category $\cC$, we denote by $\Thick(\cS)$ the smallest thick subcategory of $\cC$ containing $\cS$. Likewise, we write $\Loc(\cS)$, resp. $\Locid~(\cS)$, for the smallest localizing subcategory, resp. localizing tensor ideal, containing $\cS$.

\subsection*{Acknowledgments}
We would like to thank Andrew Blumberg, Mark Hovey, Henning Krause, and Hal Sadofsky for helpful discussions related to this work, Amnon Yekutieli for helpful comments on an earlier version of this paper, and the referee for many useful suggestions. 

A preliminary version of the results in the first two sections of the present paper was previously contained in the author's joint work \cite{bhv}, while the equivalence between $\Dcmpl(A)$ and a suitable derived category of complete modules was also considered (via different methods) in the third author's PhD thesis \cite{gab_phd}. 

The first author was supported by the Danish National Research Foundation Grant DNRF92 and the European Unions Horizon 2020 research and innovation programme under the Marie Sklodowska-Curie grant agreement No. 751794, and acknowledges the hospitality of the Newton Institute in Cambridge, UK, where part of this work was carried out. The second author thanks Haifa University for its hospitality. The third author would like to thank the Max Planck Institute for Mathematics for its hospitality.

\section{Completion for comodules}\label{sec:completion}
In this section we study the completion functor for comodules. As we shall see, this differs from the $A$-module completion functor, as in general the forgetful functor to $A$-modules does not preserve limits. Nonetheless, under suitable conditions we show that the comodule completion functor is the composite of the $A$-module completion functor and a functor $\iota$ that is defined by a certain pullback diagram, which, informally speaking, extracts the largest possible subcomodule of the $A$-module completion functor.


\subsection{Limits}
Let $(A,\Psi)$ be a Hopf algebroid. For an overview on the theory of comodules over a Hopf algebroid we refer the reader to \cite[Appendix A1]{ravenel_86}. Let $\epsilon_* \colon \Comod_{\Psi} \to \Mod_A$ be the forgetful functor from comodules to modules; our notation indicates that this is in fact the functor induced by the map of Hopf algebroids $\epsilon \colon (A,\Psi)\to(A,A)$ which is the identity on $A$ and the counit on $\Psi$. Since $\epsilon_*$ does not preserve arbitrary limits (indeed, it does not even preserve products), the existence of limits in the category of comodules is not immediate. Hovey has shown that the category of comodules is complete \cite[Prop.~1.2.2]{hovey_htptheory}, by constructing the product of a system of comodules. Following his argument, we explain briefly how to construct the inverse limit of a system of comodules. The first step is to define inverse limits for extended comodules, where an adjointness argument shows that, for an inverse system $(N_k)$ of $A$-modules, we must have
\begin{equation}\label{eq:extendedpsilim}
\psilim(\Psi \ctimes N_k) \cong \Psi \ctimes \lim_k N_k,
\end{equation}
where we write $\psilim(-)$ for the limit in $\Comod_{\Psi}$. One can then construct $\psilim(f)$, where $f$ is a map of extended comodules. For a general inverse system $(M_k)$ of comodules, there are exact sequences of comodules
\[
\xymatrix{0 \ar[r] & M_k \ar[r] & \Psi \ctimes M_k \ar[r]^{f_k} & \Psi \ctimes T_k,}
\]
where $T_k$ is the cokernel of the coaction map of $M_k$. This enables us to construct the inverse limit of $(M_k)$ as $\psilim(M_k) = \ker(\psilim(f_k))$, see~\cite[Sec.~4.1]{bhv} for details. 
\begin{lem}\label{lem:limcomoduleamod}
	Let $(M_k)$ be an inverse system of comodules, then the natural morphism of $A$-modules
	\[\xymatrix{
\tau\colon\epsilon_* \psilim(M_k) \ar[r]& \lim_k(\epsilon_*M_k),}
	\]
is an injection when $\Psi$ is a projective $A$-module. 
\end{lem}

\begin{proof}
From the discussion above we know that $\psilim(M_k)$ is given by the kernel 
	\[
\xymatrix{0 \ar[r] & \psilim(M_k) \ar[r] & \Psi \ctimes \lim_k(M_k) \ar[r] & \Psi \ctimes \lim_k (T_k),}
	\]
where $T_k$ is the cokernel of the coaction map of $M_k$. We then have a commutative diagram of $A$-modules (where we omit writing $\epsilon_*$ for simplicity)
	\[
\xymatrix@C=3em{
0 \ar[r] & \psilim(M_k) \ar@{-->}[d]_\tau \ar[r] & \Psi \ctimes \lim_k(M_k) \ar[d] \ar[r] & \ar[d] \Psi \ctimes \lim_k(T_k) \\ 
0 \ar[r] & \lim_k(M_k) \ar[r]_-{\lim_k(\psi_{M_k})} & \lim_k(\Psi \ctimes M_k) \ar[r] & \lim_k(\Psi \ctimes T_k),
}
	\]
which induces the natural map $\tau$. This is an injection provided the two right hand vertical arrows are, which we claim is true when $\Psi$ is a projective $A$-module. Indeed, it is not hard to check that the corresponding statement is true for products, and we can then write the inverse limit as a kernel of maps between the product to deduce the claimed result.  
\end{proof}

The next result shows that, under certain conditions, the inverse limit in comodules can be determined by first taking the inverse limit of the underlying modules and then extracting a subcomodule using a pullback.

\begin{prop}\label{prop:comlimaspb}
Suppose $(M_k)$ is an inverse system of comodules such that the canonical maps
\[
\xymatrix{\Psi \ctimes \lim_k(M_k) \ar[r] & \lim_k(\Psi \ctimes M_k) & & \Psi \ctimes \lim_k(\Psi \ctimes M_k) \ar[r] & \lim_k(\Psi \ctimes \Psi \ctimes M_k)}
\]
are monomorphisms, then the inverse limit in comodules, $\psilim(M_k)$, can be computed by the following pullback of $A$-modules:
\begin{equation}\label{eq:pullback2}
\xymatrix@C=4em{
\psilim(M_k) \ar[r]^-{p} \ar[d]_{\tau} & \Psi \ctimes \lim_k(M_k) \ar[d]^{j} \\
\lim_k(M_k) \ar[r]_-{\lim_k(\psi_{M_k})} & \lim_k (\Psi \ctimes M_k).
}
\end{equation}
\end{prop}

\begin{proof}
	Let $P$ be the pullback of the span part of~\eqref{eq:pullback2}. We will omit the standard verification that $P$ naturally admits the structure of an $\Psi$-comodule. Thus, it remains to show that $P$ satisfies the universal property of the limit. So, suppose we have a comodule $N$, along with compatible comodule morphisms $f_k\colon N \to M_k$ for all $k$. By the universal property of the inverse limit in $A$-modules, we obtain an $A$-module morphism $f\colon N \to \lim_k(M_k)$, and a diagram:
\[
\xymatrix@C=4em{
N \ar@/_2pc/[ddr]_-f \ar@/^2pc/^-{(1 \ctimes f)\psi_N}[drr] \ar@{-->}[rd]^g\\
&P \ar[r]^-{p} \ar[d]_i & \Psi \ctimes \lim_k(M_k) \ar[d]^j \\
&\lim_k(M_k) \ar[r]_-{\lim_k(\psi_{M_k})} & \lim_k(\Psi \ctimes M_k).
}
\]
One can check that this diagram commutes, and so we obtain a (unique) morphism $g\colon N \to P$, which can be shown to be a morphism of comodules. The morphism $\pi_i\colon P \to M_k$ is the composite of $i\colon P \to \lim_k(M_k)$ and the $A$-module projection maps; once again, these can be checked to be comodule morphisms making the required diagrams commute. 
\end{proof}
	
\begin{rem}
	Hovey realized that, under suitable conditions, the comodule product can be defined as the largest possible subcomodule of the $A$-module product; see the remark after Proposition 1.2.2 of~\cite{hovey_htptheory}. An alternative proof (under slightly more general conditions) is given in the thesis of Sitte~\cite[Lem.~3.5.12]{sitte2014local}, and our approach follows his closely. Similar ideas are contained in unpublished work of Sadofsky.
\end{rem}

\subsection{Completion}

Let $(A,\Psi)$ be a Hopf algebroid and fix a finitely generated invariant ideal $I \subseteq A$. Undecorated notation will usually refer to the module-theoretic as opposed to comodule-theoretic constructions. 

\begin{defn}
The $I$-adic completion $C^I_\Psi(M)$ of a $\Psi$-comodule $M$ is defined as $C^I_\Psi(M) = \psilim(M \ctimes A/I^k)$. Note that $C_{A}^I(\epsilon_* M) = \lim_k(\epsilon_*M \ctimes A/I^k)$ is the usual module-theoretic completion of $M$. For the sake of simplified notation, we will often write $C^I(M)$ for the latter or just $\wM$ when the ideal is clear from the context.
\end{defn}

As an application of \Cref{lem:limcomoduleamod}, we obtain the following coarse comparison between the two notions of completion.

\begin{lem}\label{lem:comoduleamod}
For $M \in \Comod_{\Psi}$ there is a natural morphism of $A$-modules
\[
\xymatrix{\tau\colon\epsilon_* C^I_\Psi(M) \ar[r] & C^I(M).} 
\]
This is an injection when $\Psi$ is a projective $A$-module. 
\end{lem}

If $(A,\Psi)$ is discrete, then $C_{\Psi}^I = C_{A}^I$, but they differ in general. In contrast to $C^I_\Psi$, the functor $C^I$ does not in general take values in the category of comodules again, because the completed coaction map takes values in a completed tensor product. In \cite{dev_cor}, Devinatz introduced a category of complete comodules over a complete Hopf algebroid to address this issue. Given $M\in\Mod_A$ let us denote its $I$-adic completion by $\wM$; we write $\cctimes$ for the completed tensor product. Suppose now that $(A,\Psi)$ is a Hopf algebroid, then for any finitely generated invariant ideal $I$, the triple $(\wA,\wPsi,I\cdot\wA)$ is a complete Hopf algebroid. 

\begin{defn}
A (left) complete $\wPsi$-comodule $M$ is a complete $\wA$-module $M$ together with a left $\wA$-linear map $\wpsi=\wpsi_M\colon M \to \wPsi \cctimes_{\wA} M$ which is counitary and coassociative. A morphism of complete $\wPsi$-comodules is, as usual, a morphism of complete modules that commutes with the structure maps. We will write $\cComod_{\wPsi}$ for the category of complete $\wPsi$-comodules. \end{defn}

Inspired by \Cref{prop:comlimaspb}, we consider the functor $\iota\colon\cComod_{\wPsi} \to \Mod_A$ defined on $N \in \cComod_\wPsi$ by the pullback diagram
\[
\xymatrix{
\iota N \ar[r]^-{p} \ar[d]_{i} & \Psi \ctimes N \ar[d]^-{j} \\
N \ar[r]_-{\wpsi_N} & \wPsi \cctimes N.}
\]
Informally speaking, $\iota N$ extracts the largest subcomodule of $N$; however, it is not clear that the map $i\colon\iota N \to N$ is injective nor that $\iota N$ admits a natural $\Psi$-comodule structure. We therefore introduce a type of Hopf algebroid for which these problems do not arise. In the next subsection, we exhibit a sufficient criterion for verifying these conditions and provide an example. 

\begin{defn}\label{def:truelevel}
A Hopf algebroid $(A,\Psi)$ is called true-level (with respect to the fixed invariant ideal $I \subseteq A$) if, for any $M \in \cComod_{\wPsi}$,  the canonical maps
\[
\xymatrix{
\Psi \ctimes M \ar[r] & \Psi \cctimes M & \mathrm{and} & \Psi \ctimes (\Psi \cctimes M) \ar[r] & \Psi \cctimes \Psi \cctimes M}
\] 
are monomorphisms. 
\end{defn}

\begin{lem}\label{lem:iotalifting}
If $(A,\Psi)$ is true-level, then $\iota$ factors as $\cComod_{\wPsi} \to \Comod_{\Psi} \xrightarrow{\epsilon_*} \Mod_A$. Furthermore, if $N$ is a complete comodule, then the natural map $\iota N \to N$ is injective. 
\end{lem}

By abuse of notation, for a true-level Hopf algebroid $(A,\Psi)$, we will denote the functor $\cComod_{\wPsi} \to \Comod_{\Psi}$ given in \Cref{lem:iotalifting} by $\iota$ as well. 

\begin{proof}
Consider the following diagram of $A$-modules
\begin{equation}\label{coactiond}
\xymatrix{
\iota N \ar@{.>}[dr]_{\rho} \ar@/^1.5pc/[drr]^{(\Delta \ctimes 1) p} \ar@/_1.5pc/[ddr]_p& & \\
&\Psi\ctimes \iota N \ar[r]^-{1 \ctimes p} \ar[d]_-{1 \ctimes i} & \Psi\ctimes\Psi\ctimes N \ar[d]^-{1 \ctimes j}  \\
&\Psi\ctimes N \ar[r]_-{1 \ctimes \wpsi_N} & \Psi\ctimes\wPsi\cctimes N,
}
\end{equation}
and note that since $\Psi$ is flat, the square part is a pullback. Hence, to obtain a candidate $\rho$ for the coaction map of $\iota N$ it is enough to show that the outer part of~\eqref{coactiond} commutes; we may do this after composing with the monomorphism $j'\colon \Psi\ctimes\Psi\cctimes N \to \wPsi\cctimes\wPsi\cctimes N$. A routine diagram chase then yields the desired commutativity. A further careful diagram chase, using the fact that $N\to \wPsi\cctimes N$ is counital and coassociative, shows that $\rho$ is indeed a coaction for $\iota N$, so that $\iota N$ is a $\Psi$-comodule. Finally, since $j$ is assumed to be injective, it follows that $i\colon \iota N \to N$ is injective.
\end{proof}

We are now ready to prove the main result of this section. 

\begin{thm}\label{thm:completionformula}
If $(A,\Psi)$ is true-level, then there is an isomorphism $C_\Psi^I \simeq \iota C^I$ of $\Psi$-comodules.
\end{thm}
\begin{proof}
Let $M \in \Comod_{\Psi}$. Specializing \Cref{prop:comlimaspb} to the tower $(M_k) = (M \ctimes A/I^k)$ yields the pullback diagram on the left
\[
\xymatrix{C_\Psi^I(M) \ar[r] \ar[d] & \Psi \ctimes C^I(M) \ar[d] & \iota C^I(M) \ar[r] \ar[d] & \Psi \ctimes C^I(M) \ar[d] \\
C^I(M) \ar[r] & C^I(\Psi \ctimes M) & C^I(M) \ar[r] & C^I(\widehat{\Psi} \ctimes C^I(M)),}
\]
while the right square is a pullback diagram by definition of $\iota$. Using the natural maps between these diagrams, the natural isomorphism 
\[
\xymatrix{C^I(M \ctimes N) \ar[r]^-{\sim} & C^I(C^I(M) \ctimes C^I(N))}
\]
of $A$-modules then furnishes a natural isomorphism $C_\Psi^I(M) \cong \iota C^I(M)$ between pullbacks of $A$-modules. Finally, the argument of \Cref{lem:iotalifting} shows that the canonical comparison isomorphism is compatible with the $\Psi$-coactions on both sides. 
\end{proof}

\subsection{Examples of true-level Hopf algebroids}

In order to provide examples of true-level Hopf algebroids, we will make use of the following result. 

\begin{prop}\label{prop:caresult}
Let $A$ be a regular local Noetherian ring, $\fm$ the maximal ideal of $A$, and $N$ an $A$-module satisfying one of the following two conditions:
\begin{enumerate}
	\item $N$ is a projective $A$-module.  
	\item $A$ is complete, $N$ is flat, and $N \to \wN$ is injective.
\end{enumerate}
If $M$ is a complete $\wA$-module, then the natural completion map
\[
\xymatrix{N \ctimes M \ar[r]^-{\eta} & N \cctimes M}
\]
is a monomorphism. 
\end{prop}

\begin{proof}
First assume Condition (1), i.e., that $N$ is a projective $A$-module. We claim that the map $N\ctimes M \to N\cctimes M$ is injective for all complete modules $M$. It suffices to consider free modules, so let $F = \bigoplus A$ be free and consider the canonical map
\[
\xymatrix{F \ctimes M \cong \bigoplus M \ar[r] & \prod  M.}
\] 
This an injection, which can be checked by forgetting down to abelian groups. Since $M$ is complete by assumption and $A/\fm^k$ is finitely presented for all $k\ge 1$, the possibly infinite product $\prod M$ is complete as well. Hence, the map above factors through a monomorphism $F\ctimes M \to F\cctimes M$. 

Now assume condition (2) of \Cref{prop:caresult} holds. We first claim that $N$ is pure in $\prod A = \prod_JA$ for some indexing set $J$. Indeed, let $E$ be the cotorsion envelope~\cite{bican_all_2001} of $N$, which is flat because $N$ is. As a special case of the main result of~\cite{enochs_flat_1984}, $E$ is thus of the form $\prod_{\p\in \Spec(A)} T_\p $, where $T_\p$ is the $\p$-completion of a free $A_{\p}$-module. It follows from \cite[Thm.~A.2(b)]{hovey_morava_1999} that $\fm$-adic completion and $L_0$-completion with respect to $\fm$ coincide on flat modules. Consequently, by \cite[Prop.~A.15]{frankland}, we know that if $N$ is flat, then $\wN$ is pro-free, i.e., $\wN=C^{\fm}(\bigoplus A)$. Furthermore, $C^{\fm}(\bigoplus A)\to \prod A$ is a split monomorphism by~\cite[Prop.~A.13]{hovey_morava_1999}, thus it suffices to show that $N \subseteq \wN$ is pure. Since $A$ is complete, it is cotorsion, and so is $\prod A$, because $\Ext$ commutes with direct products in the second variable if $A$ is Noetherian. It follows that the monomorphism $N \to \prod A$ factors through $E$, i.e.,
\[
\xymatrix{N \ar[r] & \prod_{\p\in \Spec(A)}T_\p \ar[r] & \prod A.}
\]
Note that $\m\prod_{\p\neq \m} T_\p =\prod_{\p\neq \m} T_\p $, so the only possible map from $\prod_{\p\neq \m} T_\p$ into $\prod A$ is the zero map, as $\prod A$ is complete. It follows that $T_\p$ must be trivial for all $\p\neq\m$, i.e., $T_\m $ is the cotorsion envelope of $N$. Note that $N$ is pure in $T_\m$, and since we have $N \subseteq \wN \subseteq T_\m$, $N$ is also pure in $\wN$. 

Using \cite[Prop.~4.44]{lambook} and a colimit argument, we can verify that for all $A$-modules $M$ and any indexing set $I$, the canonical map $\left( \prod_I A\right) \ctimes M \to \prod_I M$ is injective. It follows from purity that the composite $N\ctimes M \to \left( \prod A\right) \ctimes M \to \prod M$ is injective as well. But $\prod M$ is complete, so the map above factors through $N\ctimes M\to N\cctimes M$, which thus must be injective. 
\end{proof}

Applying the proposition to $N = \Psi$ and the complete comodules $M$ or $\Psi \cctimes M$ yields:

\begin{cor}
Let $(A,\Psi)$ be a Hopf algebroid with $A$ a regular local Noetherian ring. If $(A,\Psi)$ satisfies one of the following two conditions:
	\begin{enumerate}
		\item $\Psi$ is a projective $A$-module, or
		\item $A$ is a complete ring, $\Psi$ is flat, and the completion map $\Psi \to \wPsi$ is injective, 
	\end{enumerate}
then $(A,\Psi)$ is true-level. 
\end{cor}

\begin{example}
In~\cite{baker_$i_n$-local_2000} Baker studies the $I_n$-localization $\mathcal{E}(n)$ of Johnson--Wilson theory $E(n)$. In particular, $\mathcal{E}(n)_*$ is local regular Noetherian and he proves that the associated cooperations $\mathcal{E}(n)_*\mathcal{E}(n)$ form a free module over $\mathcal{E}(n)_*$. Therefore, ($\mathcal{E}(n)_*,\mathcal{E}(n)_*\mathcal{E}(n))$ is true-level. Moreover, using \cite[Thm.~C]{hs_leht}, it follows that $\Comod_{\mathcal{E}(n)_*\mathcal{E}(n)}$ is equivalent to the category of $E(n)_*E(n)$-comodules. By \cite{naumann_stack_2007}, the category of comodules over any Hopf algebroid that is Landweber exact of height $n$ is a presentation for the category of quasi-coherent sheaves over the height $n$ Lubin--Tate stack. We thus have a true-level model for the latter category.
\end{example}

\section{Derived completion}\label{sec:derivedcompletion}
In the previous section we studied $I$-adic completion on the abelian category of comodules. In this section, we work in the derived setting and consider torsion and completion functors on suitable derived categories of comodules. In the case of a discrete Hopf algebroid (i.e., in the case of $A$-modules) the derived functor of completion we construct has a well-known relationship with completion on the abelian level, where it computes the left derived functors of completion. As we shall see, this is not true for an arbitrary Hopf algebroid, and the situation is more complicated in this case. 

By applying our methods to the case of $A$-modules, we obtain alternative proofs of some structural results of Hovey and Strickland~\cite{hovey_morava_1999} about derived functors of completion for complete regular local Noetherian rings, and deduce a new criterion for $L$-completeness. 

\subsection{The stable category of comodules and derived torsion and completion}\label{sec:torcomplete} 
In this section we briefly recall the stable category of comodules as well as the basic features of derived torsion and completion that we need. We refer the reader to \cite[Sec.~4 and Sec.~5]{bhv} for more details and for proofs.

\sloppy The category of comodules $\Comod_{\Psi}$ over a flat Hopf algebroid $(A,\Psi)$ is a symmetric monoidal Grothendieck abelian category. As such, its derived $\infty$-category $\D(\Psi) = \cD(\Comod_{\Psi})$ exists; however, it has been noted \cite{hovey_htptheory, krausestablederivedcat,bhv} that it has some undesirable properties --- for example, the tensor unit $A$ need not be compact. Based on Hovey's work \cite{hovey_htptheory}, the authors constructed a stable $\infty$-category $\Stable_{\Psi}$ which is amenable to the techniques in \cite{bhv}. 

Given a small $\infty$-category $\cC$, one can construct the ind-category $\Ind(\cC)$ associated to $\cC$ which can be thought as the smallest $\infty$-category closed under filtered colimits containing $\cC$; see \cite[Sec. 2.2]{bhv} or the references therein. Let $\cG_{\Psi}$ be a set representatives of isomorphism classes of dualizable $\Psi$-comodules.

\begin{defn}
	The stable $\infty$-category of $\Psi$-comodules is defined as the ind-category of the thick subcategory of $\D(\Psi)$ generated by $\cG_\Psi$. In symbols,
	\[	\Stable_\Psi=\Ind(\Thick(\cG_\Psi)).\]
\end{defn}

The resulting stable $\infty$-category is closed symmetric monoidal and compactly generated by $\cG_\Psi$. Following \cite{hovey_htptheory}, we will need the following additional condition to guarantee that the abelian category of $\Psi$-comodules is generated by $\cG_{\Psi}$ as well.

\begin{defn}
  A Hopf algebroid $(A,\Psi)$ is said to be an Adams Hopf algebroid if $\Psi=\colim_i \Psi_i$ for some filtered system $\{\Psi_i\}$ of comodules, which are finitely generated and projective over $A$.
\end{defn}

The ring of cooperations of the ring spectra $MU$, $MSp$, $K$, $KO$, $H
\mathbb{F}_p$, and $K(n)$ are all examples of Adams Hopf algebroids \cite[Lem.~1.4.6, Thm.~1.4.7]{hovey_htptheory}.

Since $\D(\Psi)$ is cocomplete, the universal property of an ind-category  guarantees that the inclusion of $\Thick(\cG_\Psi)$ in $\D(\Psi)$ extends to a symmetric monoidal functor $\omega$ on $\Stable_\Psi$. We summarize the relationship between $\Stable_{\Psi}$ and $\D(\Psi)$ in the following proposition.

\begin{prop}\label{prop:stableprops}
	Let $(A,\Psi)$ be an Adams Hopf algebroid. Then, there is an adjunction \[
\xymatrix{\Stable_{\Psi} \ar@<0.5ex>[r]^-{\omega} &  \ar@<0.5ex>[l]^-{\iota_*} \D(\Psi)}
\]
between $\Stable_{\Psi}$ and $\D(\Psi)$, where $\omega$ is continuous and $\iota_*$ is fully faithful. In the case that $(A,A)$ is a discrete Hopf algebroid, $\Stable_{A}$ is equivalent to $\cD(A)$, the usual derived category of $A$-modules.
\end{prop}
\begin{proof}
	Everything except that $\iota_*$ is fully faithful is proven in \cite[Sec.~4]{bhv}. The fully faithfulness is proved there under some further assumptions on the Hopf algebroid. This conditions can be weakened by using the recent work of Pstr\k{a}gowski. Indeed, by \cite[Thm.~3.7 and Cor.~3.8]{piotr_synthetic} we can identify $\Stable_{\Psi}$ with the $\infty$-category of spherical sheaves of spectra on dualizable comodules, and $\D(\Psi)$ with the $\infty$-category of hypercomplete spherical sheaves of spectra. The functor $\iota_*$ can then be identified with the inclusion of hypercomplete sheaves into all spherical sheaves. 
\end{proof}

Let $I \subseteq A$ be an ideal which we assume to be generated by an  invariant regular sequence $\{x_1,\ldots,x_n\}$. These conditions ensure that $A/I$ is a $\Psi$-comodule and that its image in $\Stable_{\Psi}$ is compact and dualizable.
 \begin{defn}
  	The subcategory $\Stable_{\Psi}^{I-\mathrm{tors}} \subseteq \Stable_{\Psi}$ is defined as the localizing tensor ideal of $\Stable_{\Psi}$ generated by the compact object $A/I$. The inclusion of the category $\Stable_{\Psi}^{I-\text{tors}}$ of $I$-torsion $\Psi$-comodules into $\Stable_{\Psi}$ will be denoted $\iota_{\text{tors}}$. 
  \end{defn}  

The full subcategory $\Stable_{\Psi}^{I-\mathrm{tors}}$ is compactly generated, and hence $\iota_{\text{tors}}$ admits a right adjoint $\Gamma_I$ which is smashing, i.e.,
\[
\Gamma_I(M) \simeq \Gamma_I(A) \otimes M
\] 
for any $M \in \Stable_{\Psi}$. Therefore, we can apply the results of \cite[Sec.~2]{bhv} to obtain localization and completion adjunctions 
\[
\xymatrix{\Stable_{\Psi}^{I-\mathrm{loc}} \ar@<0.5ex>[r]^-{\iota_{\mathrm{loc}}} & \Stable_{\Psi} \ar@<0.5ex>[l]^-{L_I} \ar@<0.5ex>[r]^-{\Lambda^I} & \Stable_{\Psi}^{I-\mathrm{comp}} \ar@<0.5ex>[l]^-{\iota_{\mathrm{comp}}}
}\]
with respect to the ideal $I$. When considered as endofunctors of $\Stable_{\Psi}$, the functors $(\Gamma_I,\Lambda^I)$ define an adjoint pair, so that we have a natural equivalence 
\begin{equation}\label{eq:ldmod}
\Hom_{\Stable_{\Psi}}(\Gamma_I X,Y) \simeq \Hom_{\Stable_{\Psi}}(X,\Lambda^IY)
\end{equation}
for all $X,Y \in \Stable_{\Psi}$. Moreover, $\Gamma_I$ and $\Lambda^I$ induce mutually inverse equivalences
\[
\xymatrix{
\Stable_{\Psi}^{I-\mathrm{comp}} \ar@<1ex>[r]^-{\Gamma_I}_{\sim} & \ar@<1ex>[l]^-{\Lambda^I} \Stable_{\Psi}^{I-\mathrm{tors}}.
}
\]
The next lemma is \cite[Cor.~5.26]{bhv} with slightly weakened hypotheses.
\begin{lem}\label{lem:compformula}
	Let $(A,\Psi)$ be an Adams Hopf algebroid and assume that $I$ is a finitely generated ideal, generated by an invariant regular sequence, then we have that 
	\[
\Lambda^IM \simeq \lim_{\psi,k} A/I^k \otimes M, 
\]
for all $M\in\Stable_\Psi$.
\end{lem}
\begin{proof}
By \cite[Cor.~3.8]{piotr_synthetic} we can remove the Noetherian hypothesis from \cite[Prop.~5.24]{bhv}, of which \cite[Cor.~5.26]{bhv} is a direct corollary. 
\end{proof}
In the discrete case, there are spectral sequences computing the (co)homology of the torsion and completion functors. These are given in terms of local cohomology and local homology of $A$-modules with respect to an ideal $I \subseteq A$, denoted $H^s_I$ and $H_s^I$ respectively, for which we refer the reader to \cite{gm_localhomology} or \cite[Sec.~3.2]{bhv}. The following is then \cite[Prop.~3.20]{bhv}.
\begin{prop}\label{prop:homology_ss}
		Let $A$ be a commutative ring and $I$ a finitely generated ideal. Let $X \in \D_A$. There are strongly convergent spectral sequences of $A$-modules:
	\begin{enumerate}
		\item $E_2^{s,t}=H^{s}_I(H^tX) \implies H^{s+t}(\Gamma_I X)$, and
		\item $E^2_{s,t}=H_{s}^I(H_tX) \implies H_{s+t}(\Lambda^I X)$.
	\end{enumerate}
\end{prop}
We finish this subsection with a couple of general remarks.
\begin{rem}
	In the special case of a discrete Hopf algebroid $(A,A)$ the conditions on the ideal $I$ can be weakened. Here, instead of $A/I$ we can use a suitable Koszul complex for the construction of $\Gamma_I$ and $\Lambda^I$, see Remark 3.10 and Theorem 3.11 of \cite{bhv}.
\end{rem}
\begin{rem}
	The theory of torsion and complete objects in a suitable category $\cC$ has previously been studied by Hovey--Palmieri--Strickland \cite[Thm.~3.3.5]{hps_axiomatic} and Dwyer--Greenlees \cite{dwyer_complete_2002} among others. 
\end{rem}
\subsection{Derived completion for modules}\label{sec:lcompletion}
In the previous section we introduced a derived version of completion for the stable category associated to a suitable Hopf algebroid. In this section we focus on discrete Hopf algebroids, i.e., the derived category of $A$-modules for a commutative ring $A$. We will show how the derived completion functor $\Lambda^I$ is related to the derived functors of ordinary $I$-adic completion. 

To that end, let $A$ be a commutative ring and $I$ an ideal in $A$. Recall that the $I$-adic completion, defined by  $C^I(M)= \lim_s M/I^sM$, is neither right nor left exact in general for non-finitely generated $A$-modules, see \cite[App.~A]{hovey_morava_1999} for example. 
We are then led to consider either the left or right derived functors of completion. It turns out that if $A$ is an integral domain, then the higher right derived functors of completion vanish \cite[Sec.~5]{gm_localhomology}, and we hence focus on the left derived functors of completion. 
\begin{defn}
	Let $A$ and $I$ be as above and $M \in \Mod_A$. For $s \ge 0$, let $L_s^I(M)$ denote the $s$-th left derived functor of $I$-adic completion on $M$. 
\end{defn}
If the ideal $I$ is clear from context, then we will usually just write $L_s(M)$. 

Note that $L_0M$ is in general not equivalent to $C^I(M)$, see \Cref{cor:gmses}. There is a natural homomorphism $\eta\colon M \to L_0 M$, and $M$ is said to be $L$-complete when $\eta$ is an isomorphism. Let $\LM \subseteq \Mod_A$ be the full subcategory of $L$-complete $A$-modules. The following is proved under the assumption that $A$ is a complete local ring in \cite{hovey_morava_1999}.
\begin{prop}\label{prop:lcomplete}
	The category $\LM$ of $L$-complete modules is an abelian subcategory of $\Mod_A$. There are enough projectives in $\LM$, and each projective object is a retract of a pro-free module, i.e., the completion of a free $A$-module. Moreover, for all $k \ge 0$, the modules $L_kM$ lie in $\LM$. 
\end{prop}
\begin{proof}
	The fact that $\LM$ is an abelian subcategory of $\Mod_A$ is proved in the same way as in \cite[Thm.~A.6(e)]{hovey_morava_1999} --- the proof only uses that $L_0$ is right exact and that for any short exact sequence $M' \to M \to M''$ there is a long exact sequence 
	\[
L_{k+1}M'' \to L_kM' \to L_kM \to L_kM'' \to L_{k-1}M'. 
	\] 
	Both of these follow from the definition of the functors $L_k$ as the left derived functors of $I$-adic completion. 

	Let $F$ be a free $A$-module. Then $L_0F \cong C^I(F)$ is projective in $\LM$ since for any $M \in \LM$ there is an adjunction $\Hom_{\LM}(L_0F,M) \cong \Hom_{A}(F,M)$ and because epimorphisms in $\LM$ are epimorphisms in $\Mod_A$. To see that $\LM$ has enough projectives, let $M \in \LM$, then there exists a free $A$-module $F$ and an epimorphism $F \to M$. Since $L_0$ is right exact, this gives rise to an epimorphism $L_0F \to M$. If $M$ is itself projective in $\LM$, then this must have a section, and we see that any projective in $\LM$ is a retract of a pro-free one. 

	The final statement can be proven precisely as in \cite[Thm.~4.1]{gm_localhomology}.
	\end{proof}

However, $\LM$ is not a Grothendieck category in general, because filtered colimits are not necessarily exact.  Despite the fact that $\LM$ is not a Grothendieck category, in \Cref{sec:tilting} we will define a version of its derived category and show that it is equivalent to a certain full subcategory of the usual derived category $\D(A)$. 

In the previous subsection we introduced abstract local homology functors $\Lambda^I$ for the stable category of $\Psi$-comodules. In the case of $A$-modules $\Lambda^I$ acquires two interpretations: First, it agrees with the total left derived functor of completion. On the other hand, the homology $L_s^I(M)$ can  be computed in terms of the local homology groups $H_s^I(M)$. Both of these hold under a mild regularity assumption on the ideal $I$ called weakly proregularity, for example, if $I$ is generated by a finite regular sequence or for arbitrary ideals if $A$ is Noetherian. We summarize these results in the next theorem; for a proof of the first part, see \cite[Prop.~3.16]{bhv} - note that we can remove the boundedness assumption there using \cite[Cor.~5.25]{psy}. The second part was first proved by Greenlees and May~\cite[Thm.~2.5]{gm_localhomology}, and then extended by Schenzel \cite[Thm.~1.1]{schenzel_proreg} and Porta, Shaul, and Yekutieli \cite[Cor.~5.25]{psy}. 
\begin{thm}\label{thm:gm_localhomology}
	Let $A$ be a commutative ring and $I$ an ideal generated by a finite weakly proregular sequence. Then:
	\begin{enumerate}
		\item There is a natural equivalence $\Lambda^I\simeq \mathbf{L}C^I$.
		\item For all $A$-modules $M$, and all $s\geq 0$, there is a natural isomorphism	$H^I_s(M)\cong L_s^I(M)$.
	\end{enumerate}
\end{thm}
We now establish a new criterion for $L$-completeness, generalizing the $\Ext$-$p$-completeness criterion due to Bousfield and Kan. We start with a general lemma which the authors were unable to find in the literature as stated. We note that the same argument works for ordinary categories as well and so it could be useful in a variety of situations.  

\begin{lem}\label{lem:pushgen}
Suppose $f_!\colon \cC \to \cD$ is a continuous functor between presentable $\infty$-categories with right adjoint $f^*$ and that $\cC$ is generated by a set of objects $\cG$. If $f^*$ is conservative, then $\cD$ is generated by $f_!\cG$. If we assume additionally that every element of $\cG$ is compact and that $f^*$ admits a further right adjoint $f_*$, then  $\cD$ is compactly generated by $f_!\cG$.
\end{lem}
\begin{proof}
Suppose $\alpha\colon X \to Y$ is a morphism in $\cD$ such that $\Hom_{\cD}(f_!G,\alpha)$ is an equivalence for all $G \in \cG$. By adjunction, $\Hom_{\cC}(G,f^*\alpha)$ is an equivalence as well, so $f^*\alpha$ is an equivalence because $\cG$ generates. Since $f^*$ is conservative, $\alpha$ must be an equivalence, and it follows that $f_!\cG$ generates $\cD$. To see the last claim, note that since $f^*$ has a right adjoint, it preserves colimits, hence $f_!$ preserves compact objects and thus $f_!\cG \subseteq \cD^{\omega}$. 
\end{proof}
As they may be of independent interest, the next results are written for comodules over a flat Hopf algebroid even though we will specialize to the discrete case for the main theorem. To this end, we will need a stronger notion of an invariant ideal, which we call strongly invariant. 
\begin{defn}
	Let $(A,\Psi)$ be a flat Hopf algebroid. We call an invariant ideal $I$ strongly invariant if, for $1 \le k \le n$ and every comodule $M$ which is $(x_1,\dots,x_{k-1})$-torsion as an $A$-module, there is a comodule structure on $x_k^{-1}M$ such that the natural homomorphism $M \to x_k^{-1}M$ is a comodule morphism. 
\end{defn}
Note that in the case of a discrete Hopf algebroid this condition is automatic.
\begin{prop}
	Let $(A,\Psi)$ be a flat Hopf algebroid and $\cG_{\Psi}$ be a set of representative of isomorphism classes of dualizable $\Psi$-comodules. Suppose $I \subseteq A$ is a strongly invariant ideal in $A$ generated by a regular sequence $x_1,\dots,x_n$, then
\[
\Stable_{\Psi}^{I-\loc}=\Loc\left (\bigoplus_{i=1}^n x_i^{-1}G/(x_1,\dots, x_{i-1})\colon G\in \cG_{\Psi}\right )
\]
as subcategories of $\Stable_{\Psi}$.
\end{prop}
\begin{proof}
By \Cref{lem:pushgen} and the fact that $L_I$ is smashing, $\Stable_{\Psi}^{I-\loc}$ is the localizing subcategory in $\Stable_{\Psi}$ generated by $L_I\cG_{\Psi}$. We therefore have to show that 
\[
\Loc(L_I\cG) = \Loc\left (\bigoplus_{i=1}^n x_i^{-1}G/(x_1,\dots, x_{i-1})\colon G\in\cG_{\Psi}\right).
\]
To simplify notation, let $G \in \cG_{\Psi}$ and write $G_i = x_i^{-1}G/(x_1,\dots, x_{i-1})$. First observe that $\Gamma_I G_i = 0$, so $G_i \simeq L_IG_i$ for all $i$. Consequently, $G_i \in \Loc(L_I\cG)$ and it remains to show the other inclusion. To this end, consider the following fiber sequences of comodules
\[
\xymatrix{G/(x_1,\ldots,x_{i-1}) \ar[r] & G_i \ar[r] & G/(x_1,\ldots,x_{i-1},x_i^{\infty})}
\]
Applying $L_I$ to these sequences starting from the one for $i=n$, we  see by downward induction on $i$ that $L_IG/(x_1,\ldots,x_{i-1}) \in \Loc(\bigoplus_{i=1}^n G_i\colon G\in\cG_{\Psi})$: indeed, the case $i=n$ of the claim holds because $L_I(G/(x_1,\ldots,x_{i-1},x_i^{\infty})) = 0$. If $L_IG/(x_1,\ldots,x_{i}) \in \Loc(\bigoplus_{i=1}^n G_i\colon G\in\cG_{\Psi})$, then passing to the colimit shows that so is $L_IG/(x_1,\ldots,x_{i-1},x_{i}^{\infty})$. Since $G_i \simeq L_IG_i$, we deduce from the $i$-th fiber sequence that $L_IG/(x_1,\ldots,x_{i-1}) \in \Loc(\bigoplus_{i=1}^n G_i\colon G\in\cG_{\Psi})$ as well. Consequently, $L_IG \in \Loc(\bigoplus_{i=1}^n G_i\colon G\in\cG_{\Psi})$ for all $G \in \cG_{\Psi}$, as desired. 
\end{proof}

The next result is an immediate consequence of the characterization of $\Lambda^I$-acyclics.

\begin{cor}\label{cor:Lamdalocal}
	Let $(A,\Psi)$ be a flat Hopf algebroid and $\cG_{\Psi}$ be a set of representative of isomorphism classes of dualizable $\Psi$-comodules. Then, for every $M\in\Stable_\Psi$,
	\[
        \Lambda^IM\simeq M \Longleftrightarrow \iHom(x_i^{-1}G/(x_1,\dots, x_{i-1}),M)\simeq0 \text{ for all } 1 \le i \le n \text{ and } G \in \cG_{\Psi}.
    \]
\end{cor}

In \cite{bousfield_homotopy_1972}, Bousfield and Kan define an abelian group $A$ to be $\Ext$-$p$ complete if the map $A\to \Ext^1_\Z(\Z/p^\infty,A)$ is an isomorphism. They prove in \cite[VI.~3.4(i)]{bousfield_homotopy_1972} that this is equivalent to $\Hom_\Z(\Z[p^{-1}],A)=\Ext_{\Z}^1(\Z[p^{-1}],A)=0$. The following result is thus the natural generalization of $\Ext$-$p$ completeness for modules over any commutative ring.

\begin{thm}[$\Ext$-$I$ completeness criterion]\label{thm:Lcompletecriterion}
    Let $A$ be a commutative ring and $I\subseteq A$ an ideal generated by a regular sequence $x_1, x_2, \dots ,x_n$. If $M$ is an $A$-module, then
    \[
        M \text{ is $L$-complete } \Longleftrightarrow \Ext_A^q(x_i^{-1}A/(x_1,\dots, x_{i-1}),M)=0 \text{ for all }1 \le i \le n \text{ and all } q \ge 0.
    \]
\end{thm}
\begin{proof}
First note that we may take $\cG_{A} = \{A\}$, because dualizable $A$-modules are finitely generated and projective. By the spectral sequence of \Cref{prop:homology_ss}, a module $M$ is $L$-complete if and only if $M$ is $\Lambda^I$-local as a complex concentrated in degree zero. But by \cref{cor:Lamdalocal} in the case of a discrete Hopf algebroid, $M$ is $\Lambda^I$-local if and only if $\iHom(x_i^{-1}A/(x_1,\dots, x_{i-1}),M)\simeq 0$ for all $i=1,\dots,n$. The result thus follows after applying homology to the latter equivalence. 
\end{proof}
\begin{rem}
	Since the projective dimension of $x_i^{-1}A/(x_1,\ldots,x_{i-1})$ as an $A$-module is at most $i$, it suffices to check the vanishing of $\Ext^q_A(x_i^{-1}A/(x_1,\dots, x_{i-1}),M)$ for $0 \le q \le i$. 
\end{rem}

\subsection{Derived completion for comodules}

Suppose now that $(A,\Psi)$ is a flat Hopf algebroid and let $I \subseteq A$ be a finitely generated ideal generated by an invariant regular sequence. The following is inspired by local homology in the case of a discrete Hopf algebroid.
\begin{defn}
	Let $M\in\Stable_\Psi$. We define the $s$-th local homology of $M$ to be
	\[\Lambda^I_s(M)=H_s(\Lambda^IM).\]
\end{defn}
In this section, we will construct a Grothendieck type spectral sequence calculating local homology for any $\Psi$-comodule $M$. However, the abutment of this spectral sequence will not be related to the left or right derived functors of comodule completion $C_{\Psi}^I$ as studied in \Cref{sec:completion}. To see this we first need the following lemma, which shows that the right derived functors of the completion functor on comodules vanish under mild conditions on the ring $A$:

\begin{lem}\label{lem:comodvanishlim}
If $A$ is an integral domain and $J \in \Comod_{\Psi}$ is an injective comodule, then $\ilim_{\Psi,k} A/I^k \otimes J = 0$. In particular, all right derived functors of $\ilim_{\Psi,k} A/I^k \boxtimes - \colon \Comod_{\Psi} \to \Comod_{\Psi}$ are zero. 
\end{lem}
\begin{proof}
Since any injective comodule is a retract of an extended comodule on an injective $A$-module $J'$ \cite[Lem.~2.1(c)]{hs_localcohom}, we can assume without loss of generality that $J = \Psi \boxtimes J'$. Hence, by \eqref{eq:extendedpsilim} there is an isomorphism
\[ 
\ilim_{\Psi,k} A/I^k \otimes J \cong \Psi \boxtimes \lim_k A/I^k \boxtimes J'
\]
and we conclude by \cite[Lem.~5.1]{gm_localhomology} which gives that $\lim_k A/I^k \boxtimes J'=0$.
\end{proof}
\begin{rem}\label{rem:negativegrouops}
Recall that \Cref{thm:gm_localhomology}(1) says that, for $A$-modules, the homology groups of $\Lambda^I$ compute the left derived functors of completion. For an arbitrary Hopf algebroid, \Cref{ex:klim} below shows that, unlike the case of $A$-modules, there exist negative local homology groups. This, along with \Cref{lem:comodvanishlim}, implies that for comodules the homology groups of $\Lambda^I$ cannot, in general, be given by the partial left or right derived functors of $C^I_{\Psi}$. 
\end{rem}
\begin{prop}\label{prop:comodulecssgrothendieck}
Suppose $(A,\Psi)$ is an Adams Hopf algebroid, $I$ is a finitely generated ideal of $A$ generated by an invariant regular sequence, and $M \in D_{\Psi}$. There is a conditionally and strongly convergent spectral sequence of comodules
\begin{equation}\label{eq:comodulecompletiongrothendieckspectralsequence}
E_2^{p,q} \cong \ilim^p_{\Psi,k}\Tor^{\Psi}_q(A/I^k,M) \implies \Lambda^I_{q-p}(M).
\end{equation}
\end{prop}
\begin{proof}
Let $\cA_1$ and $\cA_2$ be abelian categories and assume that $\cA_1$ has enough injective objects. Recall that there is a conditionally convergent spectral sequence
\[
E_2^{p,q} \cong \mathbf{R}^pF(H^q(X)) \implies H^{p+q}(\mathbf{R}F(X))
\]
for any left exact functor $F\colon \cA_1 \to \cA_2$ and $X \in \Ch(\cA_1)$, see~\cite[5.7.9, Cor.~10.5.7]{weibel_homological}. We apply this spectral sequence to the functor $\lim\colon \Comod_{\Psi}^{\N} \to \Comod_{\Psi}$ and $X = (A/I^k \otimes M)_{k \in \N} \in \D(\Psi)^{\N}$, the category of towers of chain complexes of $\Psi$-comodules up to quasi-isomorphism. Switching the grading so that the spectral sequence converges to $H_{q-p}\mathbf{R}F(X)$, 
the $E_2$-page is readily identified as 
\[ 
E_2^{p,q} \cong \ilim^p_{\Psi,k}\sh_q(A/I^k \otimes M) \cong \ilim^p_{\Psi,k}\Tor^{\Psi}_q(A/I^k,M).
\]
This spectral sequence then converges conditionally to 
\[ 
\sh_{q-p}(\mathbf{R}\ilim_{\Psi,k} (A/I^k \otimes M)),
\]
which is $\Lambda^I_{q-p}(M)$ by \Cref{lem:compformula} and \cite[Lem.~4.26]{bhv}. The latter, which says that $\ilim_{\Psi,k} (A/I^k \otimes M)$ can be unambiguously interpreted as taken in either $\D(\Psi)$ or $\Stable_\Psi$, is not stated in the generality we need here, but it holds under our hypotheses in light of \Cref{prop:stableprops}. Finally, since $I$ is generated by a finite regular sequence, this spectral sequence has a horizontal vanishing line at $E_2$ and hence converges strongly. 
\end{proof}

Since $\ilim^p=0$ for $p > 1$ for any discrete Hopf algebroid, this spectral sequence degenerates to a short exact sequence. Hence we recover as a special case:

\begin{cor}[Greenlees--May]\label{cor:gmses}
If $M \in  \Mod_A$, then there is a short exact sequence 
\[\xymatrix{0 \ar[r] & \lim^1_k \Tor_{s+1}^A(A/I^k,M) \ar[r] & L_sM \ar[r] & \lim_k \Tor_s^A(A/I^k,M) \ar[r] & 0}\]
for any $s \ge 0$. In particular, there exists a natural epimorphism $L_0M \to C^I(M)$.
\end{cor}

By the Artin--Rees lemma, local homology is concentrated in degree $0$ for finitely generated modules over Noetherian rings, where it simply agrees with $I$-adic completion, see \cite[Prop. A.4]{hovey_morava_1999}. Although this is no longer the case for comodules, the next proposition provides an appropriate comodule analogue where one needs to use derived limits to capture local homology on the nonzero degrees.
\begin{prop}\label{prop:fgvanishing}
Suppose that $A$ is Noetherian. For $M \in \Comod_{\Psi}^{\omega}$ the spectral sequence \eqref{eq:comodulecompletiongrothendieckspectralsequence} collapses, yielding an isomorphism 
\[ 
\xymatrix{\Lambda_{-s}^{I}M \ar[r]^-{\simeq} & {\ilim}_{\Psi,k}^{s}M \boxtimes A/I^k}
\]
of comodules for all $s\ge 0$.
\end{prop}
\begin{proof}
It is enough to show that
\[
\ilim_{\Psi,k}^s \Tor_t^{\Psi}(M, A/I^k) = 0
\]
for all $s \ge 0$ and all $t \ge 1$, so that the strongly convergent spectral sequence \eqref{eq:comodulecompletiongrothendieckspectralsequence} collapses. In fact, we will prove the stronger claim that the tower $(\Tor_t^{\Psi}(M, A/I^k))_k$ is pro-trivial for all $t \ge 1$; see \cite[Lem.~1.11]{jannsen_cec} for a proof that this implies that all derived functors of $\ilim$ vanish. On the one hand, using Lurie's result~\cite[Not.~5.2.18]{dag12}, it follows that the left derived functors $\mathbf{L}_tC$ of the right exact functor
\[
\xymatrix{C\colon \Comod_{\Psi} \ar[r] & \Comod_{\Psi}^{\N},\ M \mapsto (M \boxtimes_{A} A/I^k)_k}
\] 
are pro-trivial on compact $\Psi$-modules $M$ whenever $t \ge 1$. On the other hand, it is easy to see that there is an isomorphism of towers
\[
\mathbf{L}_tC(M) \cong (\Tor_t^{\Psi}(M,A/I^k))_k
\]
for all $M \in \Comod_{\Psi}$, giving the claim.
\end{proof}
\begin{exmp}\label{ex:klim}
	For $p>2$, let $K$ be $p$-complete $K$-theory and consider local homology $\Lambda^{(p)}$ with respect to the ideal $(p) \subset K_0 = \Z_p$. Let $g$ be a topological generator of $\Z_p^{\times}$; using the fiber sequence
	\[
	\xymatrix{L_{K(1)}S^0 \ar[r] & K \ar[r]^{\psi^g-1} & K}
	\]
	one can then calculate the local homology groups of $K_*$:
	\[
	\Lambda_{-s}^{(p)}(K_*) \cong \lim^s_{K_*K,i} K_*/{p^i} \cong 
	\begin{cases}
	 K_* & \text{if } s=0 \\
	V_{\Q} & \text{if } s=1 \\
	 0 & \text{otherwise,}
	\end{cases}
	\]
	where $V_{\Q}$ is an uncountable rational vector space. For the details of this computation, see~\cite{ctc}. From this, one can formally deduce that 
	\[
	\lim_{K_*K}^s(\ldots \xrightarrow{p} K_* \xrightarrow{p} K_* \xrightarrow{p} K_*) \cong 
	\begin{cases}
	V_{\Q} & \text{if } s=2 \\
	 0 & \text{otherwise.}
	\end{cases}
	\] 
	This example implies that even if the inverse limit of comodules can be computed in terms of the inverse limit of the underlying modules and a functor $\iota$ as in \Cref{sec:completion}, in general there cannot exist a convergent Grothendieck spectral sequence associated to the composite $\lim_{\Psi} \cong \iota\circ \lim_{A}$. Indeed, using a graded version of the $\Ext$-$p$-completeness criterion of \cite{bousfield_homotopy_1972} or \Cref{thm:Lcompletecriterion}, we see that $\lim_{K_*}^s(\ldots \xrightarrow{p} K_* \xrightarrow{p} K_* \xrightarrow{p} K_*) = 0$ for all $s\ge 0$, because $K_*$ is $p$-complete. Therefore, the $E_2$-page of the composite functor spectral sequence would have to be zero, while the abutment is nontrivial. 
	\end{exmp}

\section{Tilting and $t$-structures}\label{sec:tilting}

\subsection{Derived categories and tilting}

Throughout this section we let $(A,\Psi)$ be an Adams Hopf algebroid and $\D(\Psi)=\D(\Comod_{\Psi})$ its derived $\infty$-category. We recall from \Cref{prop:stableprops} that the $\infty$-category $\Stable_{\Psi}$ is  related to $\D(\Psi)$ via an adjunction
\[
\xymatrix{\omega \colon \Stable_{\Psi} \ar@<0.5ex>[r] & \ar@<0.5ex>[l] \D(\Psi) \colon \iota_*.}
\]
Here the symmetric monoidal, continuous functor $\omega$ is given by inverting the homology isomorphisms, and the right adjoint $\iota_*$ is a fully faithful embedding. Given a comodule $M \in \Comod_{\Psi}$ we can think of it as a complex in $\D(\Psi)$ concentrated in degree 0. Via $\iota_*$ we can also consider $M$ as an object of $\Stable_{\Psi}$.

In contrast to $\Stable_{\Psi}$, the dualizable comodules are not necessarily compact in $\D(\Psi)$. Nonetheless, we claim that the smallest localizing subcategory containing them is all of $\D(\Psi)$, so that they still form a suitable collection of generators. 
\begin{lem}\label{prop:dpsigenerators}
 	The smallest localizing subcategory of $\D(\Psi)$ containing the set $\cG_{\Psi}$ of dualizable comodules is all of $\D(\Psi)$, i.e., $\cG_{\Psi}$ generates $\D(\Psi)$. 
 \end{lem} 
 \begin{proof}
 	 Since $\D(\Psi)$  is presentable, the statement of the lemma is equivalent to the claim that $Z \in \D(\Psi) \simeq 0$ if and only if $\Hom_{\D(\Psi)}(G,Z)$ is contractible for each each $G \in \cG_{\Psi}$, or equivalently that a morphism $ \phi \colon X \to Y$ is an equivalence if and only if $\Hom_{\D(\Psi)}(G,\phi)$ is an equivalence of mapping spectra for all $G \in \cal{G}_{\Psi}$. To see this, apply \Cref{lem:pushgen} with $f_! = \omega$ and $f^* = \iota_*$ --- note that $\iota_*$ is fully faithful, and so in particular conservative, so that the conditions of the lemma are satisfied. 
 \end{proof}
 
Recall that given $I \subseteq A$ an invariant regular ideal, we defined the $I$-torsion category $\Stable^{I-\mathrm{tors}}_{\Psi}$ as $\Loc^{\otimes}_{\Stable_{\Psi}}(A/I)$. We can make the same definition in $\D(\Psi)$. 
\begin{defn}\label{defn:Dtors}
	The category $\Dtors(\Psi)$ is defined as the localizing tensor ideal in $\D(\Psi)$ generated by $A/I$, i.e., 
	\[
	\Dtors(\Psi) = \Loc^{\otimes}_{\D(\Psi)}(A/I). 
	\]
\end{defn}
 
\begin{lem}\label{lem:localomega}
	If $M \in \Stable_{\Psi}^{I-\mathrm{tors}}$, then $\omega M \in \Dtors(\Psi)$. 
\end{lem}
\begin{proof}
We will prove the more general statement that for any collection $\mathrm{C}$ of objects in $\Stable_{\Psi}$ we have
	\[
\omega \Loc_{\Stable_{\Psi}}^{\otimes}(\mathrm{C}) \subseteq \Loc^{\otimes}_{\cD(\Psi)}(\omega \mathrm{C}).
	\]
To this end, first observe that $\Loc^{\otimes}_{\Stable_{\Psi}}(\mathrm{C}) = \Loc_{\Stable_{\Psi}}(\mathrm{C} \otimes \iota_*\cG_{\Psi})$ where $\cG_{\Psi}$ denotes, as usual, the set of dualizable comodules. Since $\omega $ preserves colimits and is symmetric monoidal, a standard argument shows that 
\[
\omega \Loc_{\Stable_{\Psi}}(\mathrm{C} \otimes \iota_*\cG_{\Psi}) \subseteq \Loc_{D(\Psi)}(\omega (\mathrm{C} \otimes \iota_*\cG_{\Psi})) = \Loc_{\D(\Psi)}(\omega \mathrm{C} \otimes \cG_{\Psi}).
\]
By \Cref{prop:dpsigenerators} we see that $\Loc_{\D(\Psi)}(\omega \mathrm{C} \otimes \cG_{\Psi})  = \Loc^{\otimes}_{\D(\Psi)}(\omega \mathrm{C})$.
\end{proof}

In \cite[Sec.~5]{bhv} we studied the abelian category $\Comod_{\Psi}^{I-\mathrm{tors}}$ of $I$-torsion comodules. These are, by definition, those comodules whose underlying $A$-module is $I$-torsion. Recall from \cite[Prop.~5.10]{bhv} that  $\Comod_{\Psi}^{I-\mathrm{tors}}$ is a Grothendieck abelian category, so we can define its derived $\infty$-category $\cD(\Comod_{\Psi}^{I-\mathrm{tors}})$. We can also describe the generators in $\Comod_{\Psi}^{I-\mathrm{tors}}$. In what follows, we let $T_I^{\Psi}$ denote the $I$-torsion functor on $\Comod_{\Psi}$ and write $T_I^A$ for the corresponding functor on $\Mod_A$.
\begin{lem}
	The collection $\{ G \otimes A/I^k \mid G \in \cal{G}_{\Psi}, k \ge 1 \}$ is a set of generators for $\Comod_{\Psi}^{I-\mathrm{tors}}$. 
\end{lem}
\begin{proof}
	Let $f \colon M \to N$ be a morphism in $\Comod_{\Psi}^{I-\mathrm{tors}}$. We must show that if $\Hom_{\Comod_{\Psi}^{I-\mathrm{tors}}}(G \otimes A/I^k,f) = 0$ for all $G \in \cG_{\Psi}$ and $k \ge 1$, then $f = 0$. Since the inclusion is fully faithful, we have $\Hom_{\Comod_{\Psi}}(G \otimes A/I^k,f) = 0$, or by adjunction that $\Hom_{\Comod_{\Psi}}(G,\iHom_{\Psi}(A/I^k,f)) = 0$. Since the collection of dualizable comodules generates $\Comod_{\Psi}$ we deduce that $\iHom_{\Psi}(A/I^k,f) =0$ for all $k \ge 1$. Taking colimits, we see that $\colim_k \iHom_{\Psi}(A/I^k,f) = 0$. But by \cite[Lem.~5.5]{bhv} the $I$-torsion functor $T_I^{\Psi}(-)\cong \colim_k \iHom_{\Psi}(A/I^k,-)$. Since $f$ is a morphism between $I$-torsion comodules, we deduce that $T_I^{\Psi}(f) = f = 0$, as required.  
\end{proof}

Consider then the thick subcategory $\Thick_{\cD(\Comod_{\Psi}^{I-\mathrm{tors}})}(\{ G_{\Psi} \otimes A/I^k | k \ge 1 \})$. An argument similar to \cite[Lem.~5.13]{bhv} shows that if $I$ is generated by a finite invariant regular sequence, then $\Thick_{\cD(\Comod_{\Psi}^{I-\mathrm{tors}})}(\{ G_{\Psi} \otimes A/I^k | k \ge 1 \}) = \Thick_{\cD(\Comod_{\Psi}^{I-\mathrm{tors}})}(G_{\Psi} \otimes A/I)$. 
This justifies the following definition of the $I$-torsion analogue of $\Stable_{\Psi}$.
\begin{defn}
	We define the stable category of $I$-torsion comodules by
	\[
	\begin{split}
	\Stable(\Comod_{\Psi}^{I-\mathrm{tors}}) &= \Ind\Thick_{\cD(\Comod_{\Psi}^{I-\mathrm{tors}})}(\cG_{\Psi} \otimes A/I)\\
	&=\Ind\Thick_{\cD^+(\Comod_{\Psi}^{I-\mathrm{tors}})}(\cG_{\Psi} \otimes A/I),
	\end{split}
	\]
	where $\cD^+(\Comod_{\Psi}^{I-\mathrm{tors}})$ denotes the left bounded derived category of $\Comod_{\Psi}^{I-\mathrm{tors}}$ \cite[Var.~1.3.2.8]{ha}.
\end{defn}

There is a version of the local cohomology spectral sequence for $\Gamma_I^{\Psi}$. Recall that since the forgetful functor $\epsilon_*:\Comod_{\Psi}\to\Mod_A$ is exact, we obtain a functor $\epsilon_*:\Stable_\Psi\to\D(A)$ which we also denote by $\epsilon_*$. We stress that the following construction only gives a spectral sequence of $A$-modules --- we do not know if it can be given the structure of a spectral sequence of comodules.

\begin{lem}\label{lem:sscomodules}
	For any $X \in \Stable_{\Psi}$ there is a strongly convergent spectral sequence of $A$-modules
	\[
E_2^{s,t} \cong \mathbf{R}^sT_I^{\Psi}(H^t(X)) \implies H^{s+t}(\Gamma_I^{\Psi}X). 
	\]
\end{lem}
\begin{proof}
	For $\epsilon_*X$ there is a strongly convergent spectral sequence of $A$-modules
\[
E_2^{s,t} \cong H_I^s(H^t(\epsilon_*X)) \implies H^{s+t}(\Gamma_I^A(\epsilon_*X)),
\]
see \Cref{prop:homology_ss}. We will identify this with the claimed spectral sequence. 

The local cohomology groups $H_I^s$ are equivalent to $\mathbf{R}^sT_I^A$, the derived functor of torsion in $A$-modules. We then have isomorphisms
\[
\mathbf{R}^sT_{I}^A(H^t(\epsilon_*X)) \cong \mathbf{R}^sT_{I}^A(\epsilon_* H^t(X)) \cong \epsilon_*\mathbf{R}^sT_I^{\Psi}(H^t(X)),
\]
where the last equivalence is \cite[Lem.~5.12]{bhv}. This identifies the $E_2$-page of the spectral sequence of the lemma. 

For the abutment, we have $H^{s+t}(\Gamma_I^{A}(\epsilon_*X)) \cong H^{s+t}(\epsilon_*\Gamma_I^{\Psi}(X)) \cong \epsilon_*H^{s+t}(\Gamma_I^{\Psi}(X))$, where the first isomorphism follows from \cite[Lem.~5.20]{bhv}. 
\end{proof}

After collecting this preliminary material, we can now state the main theorem of this section.

\begin{thm}\label{thm:derder}
Let $(A,\Psi)$ be an Adams Hopf algebroid and $I \subseteq A$ a finitely generated invariant ideal. 
	\begin{enumerate}
			\item Suppose that $I$ is generated by a weakly proregular sequence. If $(A,\Psi)=(A,A)$ is discrete, then there is a canonical equivalence 
		between the right completion of $\cD^{-}(\LM)$ and $\Dcmpl(A)$. Moreover, an object $M \in \cD(A)$ is $I$-complete if and only if the homology groups $H_*M$ are $L$-complete.
		\item If $I$ is generated by a regular sequence, then there is a canonical equivalence $\cD(\Comod_{\Psi}^{I-\mathrm{tors}}) \simeq \Dtors(\Psi)$. Moreover, an object $M \in \D(\Psi)$ is $I$-torsion if and only if the homology groups $H_*M$ are $I$-torsion. 

	\end{enumerate}
\end{thm}

Combining this result with \Cref{sec:derivedcompletion} yields the following tilting-theoretic interpretation of local duality.

\begin{cor}\label{cor:tilting}
For any commutative ring $A$ and $I \subseteq A$ a finitely generated ideal, the functors
\[
\xymatrix{L_0^I\colon \Mod_{A}^{I-\mathrm{tors}} \ar@<0.5ex>[r] & \LM \noloc T_I^A \ar@<0.5ex>[l]}
\]
induce mutual inverse symmetric monoidal equivalences
\[
\xymatrix{\Lambda^I\colon \cD(\Mod_{A}^{I-\mathrm{tors}}) \ar@<0.5ex>[r]^-{\sim} & \cD(\LM)\noloc \Gamma_I, \ar@<0.5ex>[l]^-{\sim}}
\]
where $ \cD(\LM)$ denotes the right completion of $\cD^{-}(\LM)$. 
\end{cor}

Transferring the standard $t$-structure on $\cD(\LM)$ via the equivalence of \Cref{cor:tilting} induces a nonstandard $t$-structure on $\cD(\Mod_{A}^{I-\mathrm{tors}})$ whose heart is the abelian category of $L$-complete $A$-modules. Since the latter category is usually not equivalent to $\Mod_{A}^{I-\mathrm{tors}}$, \Cref{cor:tilting} is a non-trivial instance of a tilting equivalence. 

\begin{rem}
Due to the phenomena discussed in \Cref{sec:derivedcompletion} (see, for example \Cref{rem:negativegrouops}), we do not have an analogue of \Cref{thm:derder}(1) for non-discrete Hopf algebroids, and thus are currently unable to prove a version of  \Cref{cor:tilting} in this generality. 
\end{rem}

As we will see in the next subsections, the proof of \Cref{thm:derder}(2) is actually a consequence of an analogous statement for $\Stable_{\Psi}$ in place of the usual derived category. 

\subsection{$t$-structures and complete modules}

We start this subsection by recalling some material about $t$-structures. 
This concept was introduced by Beilinson--Bernstein--Deligne \cite{bbd}. We follow more closely the treatment given by \cite{ha} --- namely we work with homological indexing, so that $X[n]$ denotes the $n$-fold suspension $\Sigma^nX$. 

\begin{defn}
 A $t$-structure on a triangulated category $\cD$ consists of a pair of full subcategories $(\cD_{\ge 0}, \cD_{\le 0})$ such that
\begin{enumerate}
	\item $\cD_{\ge 0}[1] \subseteq \cD_{\ge 0}$ and $\cD_{\le 0}[-1] \subseteq \cD_{\le 0}$.
	\item For $X \in \cD_{\ge 0}$ and $Y \in \cD_{\le 0}[-1]$, we have $\Hom_{\cD}(X,Y) = 0$.
	\item If $X \in \cD$, then there is a triangle
	\[
A \to X \to B \to X[1]
	\]
	with $A \in \cD_{\ge 0}$ and $B \in \cD_{\le 0}[-1]$. 
\end{enumerate}
A $t$-structure on a stable $\infty$-category $\cC$ is then a $t$-structure on the associated homotopy category of $\cC$. 
 \end{defn} 

Given a $t$-structure on a stable $\infty$-category $\cC$ we can define the heart $\cC^{\heartsuit} = \cD_{\ge 0} \cap \cD_{\le 0} \subseteq \cC$.  We let $j \colon \cC^{\heartsuit} \to \cC$ denote the inclusion functor and define $\cC_{\ge n} = \cC_{\ge 0}[n]$ and $\cC_{\le n} = \cC_{0}[n]$. By \cite[Cor.~1.2.1.6]{ha} the inclusion $\cC_{\le n} \to \cC$ has a left adjoint $\tau_{\le n}$ and $\cC_{\ge n} \to \cC$ has a right adjoint $\tau_{\ge n}$.

Given a stable category $\cC$ equipped with a $t$-structure, the left completion of $\cC$ is defined to be the limit of the tower 
\[
\xymatrix{\cdots \ar[r] & \cC_{\le 2} \ar[r]^{\tau_{\le 1}} & \cC_{\le 1} \ar[r]^{\tau_{\le 0}} & \cC_{\le 0} \ar[r]^{\tau_{\le -1}} & \cdots}.
\]
We say that the $t$-structure on $\cC$ is left complete if $\cC$ is equivalent to its left completion. We can similarly define the right completion, and a right complete $t$-structure. 

The following lemma, proved in \cite[1.3.19]{bbd}, will be useful for constructing $t$-structures on full subcategories of $\cC$.
\begin{lem}\label{lem:tstructureinduced}
	Let $(\cD_{\ge 0},\cD_{\le 0})$ be a $t$-structure on $\cC$ with heart $\cC^{\heartsuit}$. Let $\cal{S} \subseteq \cC$ be a full stable subcategory of $\cC$. If $\tau_{\ge 0}M$ and $\tau_{\le 0}M$ are in $\cal{S}$ whenever $M$ is, then $(\cal{S} \cap \cD_{\ge 0}, \cal{S} \cap \cD_{\le 0})$ defines a $t$-structure on $\cal{S}$ with heart $\cal{S} \cap \cC^{\heartsuit}$. Moreover, the truncation functors for the induced $t$-structure on $\cal{S}$ are the same as those for the $t$-structure on $\cal{D}$. 
\end{lem}

We now begin the proof of \Cref{thm:derder}(1). We always assume that the ideal $I$ is generated by a weakly proregular sequence, so that \Cref{thm:gm_localhomology} applies. The category $\LM$ of $L$-complete $A$-modules is abelian; however it is not Grothendieck abelian in general because direct sums and filtered colimits are not exact. There exist enough projectives in $\LM$ by \Cref{prop:lcomplete} and so by \cite[Sec.~1.3.2]{ha} we can associate to it the right bounded derived category $\cD^{-}(\LM)$. This category comes equipped with a natural left complete $t$-structure whose heart is equivalent to $\LM$, see \cite[Prop.~1.3.2.19]{ha} and \cite[Prop.~1.3.3.16]{ha}.

 The following results show that the right completion of this bounded derived category is naturally equivalent to the derived category $\Dcmpl(A)$ of $I$-complete $A$-modules constructed abstractly in \Cref{sec:torcomplete}. We assume that $A$ and $I$ satisfy the conditions of \Cref{thm:gm_localhomology} so that local homology computes the derived functors of completion. 
\begin{prop}\label{prop:ffcomp}
  There is a $t$-structure on $\cD^{-}(\LM)$ along with a fully faithful $t$-exact inclusion $\theta \colon \cD^{-}(\LM) \hookrightarrow \Dcmpl(A)$, whose image consists of the right bounded objects of $\Dcmpl(A)$, i.e., $\bigcup (\Dcmpl(A))_{\ge n}$.
\end{prop}
\begin{proof}
	We begin by observing that the standard $t$-structure on $\cD(A)$ is left and right complete. Indeed, $\cD(A)$ is right complete by \cite[Prop.~1.3.5.21]{ha}, while left completeness follows, for example, from \cite[Prop.~7.1.1.13]{ha} and the equivalence $\Mod_{HA} \simeq \cD(A)$ \cite[Prop.~7.1.1.15 and Rem.~7.1.1.16]{ha}. The local homology spectral sequence of \Cref{prop:homology_ss} shows that if $M \in \cD(A)$ is $I$-complete, then so are the truncations $\tau_{\ge n}M$ and $\tau_{\le n}M$. It follows from \Cref{lem:tstructureinduced} that there is an induced $t$-structure on $\Dcmpl(A)$. 

	We claim this induced $t$-structure is both left and right complete. Indeed, recall that the truncation functors on $\Dcmpl(A)$ are the restriction of the truncation functors of $\cD(A)$. Since limits in $\Dcmpl(A)$ are the same as those in $\cD(A)$, we easily see that left completeness of the induced $t$-structure follows from left completeness of the $t$-structure on $\cD(A)$. On the other hand, the colimit in $\Dcmpl(A)$ is not the same as that in $\cD(A)$ --- it is given by first taking the colimit in $\cD(A)$, and then applying $\Lambda^I$. However, by right completeness of the $t$-structure on $\cD(A)$ we already have that $M \simeq \colim_k \tau_{\ge k}M$,
	for any $M \in \Dcmpl(A)$, where the colimit is taken over the maps $\tau_{\ge k}M \xr{\tau_{\ge (k-1)}} \tau_{\ge (k-1)}M$.  It follows that the induced $t$-structure on $\Dcmpl(A)$ is right complete.

	Since the induced $t$-structure on $\Dcmpl(A)$ has heart $\LM$, applying \cite[Prop.~1.3.3.7]{ha} we deduce the existence of a $t$-exact functor $\theta \colon \cD^{-}(\LM) \to \Dcmpl(A)$. The same proposition shows that $\theta$ is fully faithful if and only if for each pair $X,Y \in \LM$ with $X$ projective, the groups $\Ext^i_{\Dcmpl(A)}(X,Y) = 0$ for $i > 0$. By the characterization of projectives in \Cref{prop:lcomplete} we have that $X$ is a retract of $L_0F$ for some free $A$-module $F$, and so we can assume $X$ has this form. For a free module $F$, we have that $L_iF \simeq 0$ for $i > 0$, as $L_i$ can be computed by taking a projective resolution. Thus, the local homology spectral sequence shows that $\Lambda^IF \simeq L_0F$, concentrated in degree 0. It follows that 
  \[
\Ext^i_{\Dcmpl(A)}(X,Y) \cong \Ext^i_{\Dcmpl(A)}(\Lambda^IF,Y) \cong \Ext^i_{\cD(A)}(F,Y) \cong \Ext_A^i(F,Y)
  \]
as $F,Y$ are just $A$-modules. Since $F$ is a free $A$-module, this is zero for $i > 0$. 
\end{proof}

\begin{proof}[Proof of \Cref{thm:derder}(1)]
	Since $\Dcmpl(A)$ is right complete, the previous proposition and \cite[Rem.~1.2.1.18]{ha} show that there is a canonical equivalence between the right completion of $\cD^{-}(\LM)$ and $\Dcmpl(A)$.

	For the second part suppose that $M\in\Dcmpl(A)$, i.e., that  $\Lambda^IM\simeq M$. The spectral sequence of \Cref{prop:homology_ss} converging to $H_*(\Lambda^IM)$ has $E^2$ page in $\LM$. Since the spectral sequence has a horizontal vanishing line and $\LM$ is abelian by \Cref{prop:lcomplete}, the abutment is in the latter category as well. 
	For the converse, let $M\in\cD(A)$ be a complex whose homology is $L$-complete. The aforementioned spectral sequence collapses to give an isomorphism $H_*(\Lambda^IM)\cong H_*(M)$. This implies that the natural map $M \to \Lambda^IM$ is a quasi-isomorphism, from where it follows that $\Lambda^I M\simeq M$.
\end{proof}

\subsection{Torsion comodules}

A similar argument as given in the previous section for complete modules also works for torsion comodules using the left bounded derived category $\cD^{+}(\Comod_{\Psi}^{I-\mathrm{tors}})$ \cite[Var.~1.3.2.8]{ha}. By the dual of \cite[Prop.~1.3.5.24]{ha} the left bounded derived category of $\Comod_{\Psi}^{I-\mathrm{tors}}$ can be identified as the full subcategory of $\cD(\Comod_{\Psi}^{I-\mathrm{tors}})$ spanned by the left bounded objects (where we equip $\cD(\Comod_{\Psi}^{I-\mathrm{tors}})$ with the standard $t$-structure \cite[Def.~1.3.5.16 and Prop.~1.3.5.18]{ha}.)

In order to prove \Cref{thm:derder}(2) we first need to introduce another category.
\begin{defn}
We denote by $\Dhtors({\Psi})$ the subcategory of complexes of comodules with cohomology in $\Comod^{I-\mathrm{tors}}_{\Psi}$.
\end{defn}
In the case of a discrete Hopf algebroid $(A,A)$ it is easy to identify $\Dtors(A)$ with $\Dhtors(A)$.  
\begin{prop}\label{prop:torHomo}
	Let $A$ be a commutative ring and $I$ a finitely generated ideal. There is an equivalence of categories
	\[
\Dhtors(A) \simeq \Dtors(A). 
	\]
\end{prop}
\begin{proof}
Recalling again that $\Comod_{\Psi}^{I-\mathrm{tors}}$ is Grothendieck abelian, this is the same argument as in the second part of the proof of \Cref{thm:derder}(1) given above.
\end{proof}

The case of an arbitrary Hopf algebroid is more difficult, and involves passing to the larger category $\Stable_{\Psi}$. The reason for this is that we do not know how to construct the local cohomology spectral sequence of \Cref{lem:sscomodules} in $\cD(\Psi)$. 
\begin{prop}\label{prop:cohomequivalence}
	There is an equivalence of categories 
	\[
\Dhtors(\Psi) \simeq \Dtors(\Psi). 
	\]
\end{prop}
\begin{proof}
	We first show that $\Dhtors(\Psi)$ is a localizing subcategory of $\cD(\Psi)$. Indeed, it is clear that it is closed under desuspension, and since $\Comod_{\Psi}$ is Grothendieck abelian and $\Comod_{\Psi}^{I-\mathrm{tors}}$ is closed under colimits, $\Dhtors(\Psi)$ is also closed under colimits. 

	Now suppose that $M \in \Dhtors(\Psi)$. We claim that $M \otimes G \in \Dhtors(\Psi)$ for each dualizable $\Psi$-comodule $G$. First observe that $H^*(M \otimes G) \cong H^*(M) \otimes G$. Then, since $G$ is dualizable, we have (for example, using \cite[Lem.~5.5]{bhv}) $T_I^{\Psi}(H^*M \otimes G) \cong T_I^{\Psi}(H^*M) \otimes G \cong H^*M \otimes G$. It follows that $M \otimes G \in \Dhtors(\Psi)$ as claimed.

	Since $A/I \in \Dhtors(\Psi)$ the previous paragraphs show that $\Dhtors(\Psi)$ is a localizing subcategory of $\cD(\Psi)$ containing $A/I \otimes G$ for each dualizable comodule $G$; in particular, we deduce that $\Dtors(\Psi) \subseteq \Dhtors(\Psi)$

	Now suppose that $X \in \Dhtors(\Psi)$, so that $\iota_* X \in \Stable_{\Psi}$. By the construction of the category $\Dhtors(\Psi)$ the local cohomology spectral sequence of \Cref{lem:sscomodules} collapses for $\iota_* X$ to show that $\Gamma_I^{\Psi}\iota_*X \to \iota_*X$ is a cohomology equivalence. It follows that there is a quasi-isomorphism $\omega  \Gamma_{I}^{\Psi}\iota_*X \xr{\sim} \omega  \iota_*X \simeq X$, where the last equivalence follows because $\iota_*$ is fully faithful. By definition $\Gamma_I^{\Psi}\iota_*X \in \Stable_{\Psi}^{I-\mathrm{tors}}$, so that $\omega \Gamma_I^{\Psi}\iota_*X \simeq X \in \Dtors(\Psi)$ by \Cref{lem:localomega}. 
\end{proof}

With this in mind, we can prove a bounded version of \Cref{thm:derder}(2).

\begin{prop}\label{prop:ff}
	 There is a $t$-structure on $\cD^+(\Comod_{\Psi}^{I-\mathrm{tors}})$ along with a fully faithful functor $\theta \colon \cD^+(\Comod_{\Psi}^{I-\mathrm{tors}}) \hookrightarrow \Dtors(\Psi)$, whose image consists of the left bounded objects of $\Dtors(\Psi)$, i.e., $\bigcup (\Dtors(\Psi))_{\le n}$.
\end{prop}
\begin{proof}
	Using \Cref{prop:cohomequivalence} we see that if $M \in \Dtors(\Psi)$, then so are the truncations $\tau_{\ge 0}M$ and $\tau_{\le 0}M$. It follows from  \Cref{lem:tstructureinduced} that there is an induced $t$-structure on $\Dtors(\Psi)$ with heart equivalent to $\Comod_{\Psi}^{I-\mathrm{tors}}$, such that the inclusion $\Dtors(\Psi) \hookrightarrow \D(\Psi)$ is $t$-exact. The induced $t$-structure is right complete because colimits in $\Dtors(\Psi)$ are computed in $\D(\Psi)$, which is right complete. By the duals of \cite[Thm.~1.3.3.2 and Rem.~1.3.3.6]{ha} we deduce the existence of a $t$-exact functor $\theta \colon \cD^{+}(\Comod_{\Psi}^{I-\mathrm{tors}}) \to \Dtors(\Psi)$, see also \cite[Rem.~1.3.5.23]{ha}. 

	Let $J$ be an injective object in the category of $I$-torsion $\Psi$-modules. Dualizing \cite[Prop.~1.3.3.7]{ha} we deduce that if $\pi_n\Hom_{\Dtors(\Psi)}(X,J) \cong \pi_n\Hom_{\cD(\Psi)}(X,J) = 0$ for all $n < 0$ and each $I$-torsion $\Psi$-comodule $X$, then $\theta$ is fully faithful, with essential image the full subcategory of left bounded objects of $\Dtors(\Psi)$. It is not hard to check that any such $J$ is a retract of $T^{\Psi}_I(L)$, for some injective $\Psi$-comodule $L$. Moreover, any such $L$ is a retract of $\Psi \otimes Q$ for some injective $A$-module $Q$ \cite[Lem.~2.1(c)]{hs_localcohom}, and so we can assume that $J$ has the form $T^{\Psi}_I(\Psi \otimes Q)$. But an adjointness argument shows that there is an equivalence $T^{\Psi}_I(\Psi \otimes Q) \cong \Psi \otimes T_I^A(Q)$. Since $Q$ is an injective $A$-module, so is $T_I^A(Q)$ \cite[Prop.~2.1.4]{broadmann_sharp}. It follows that $T^{\Psi}_I(\Psi \otimes Q)$ is an injective $\Psi$-comodule. It is then clear that $\pi_n\Hom_{\cD(\Psi)}(X,T^{\Psi}_I(\Psi \otimes Q)) = 0$ for $n < 0$, because this group is isomorphic to $\Ext^{-n}_{\Psi}(X,T^\Psi_I(\Psi \otimes Q))$, and $\Ext$ in  comodules can be computed via an injective resolution of the second variable.  
\end{proof}

In order to prove an unbounded version of the previous proposition, we first prove a stable analogue of \Cref{thm:derder}(2). 

\begin{prop}\label{prop:stabletor}
There is a canonical equivalence $\Stable(\Comod_{\Psi}^{I-\mathrm{tors}}) \simeq \Stable_{\Psi}^{I-\mathrm{tors}}$.
\end{prop}
\begin{proof}
Note that by \Cref{prop:ff} the natural functor $\cD^+(\Comod_{\Psi}^{I-\mathrm{tors}}) \to \cD^+(\Psi)$ is fully faithful. It follows that there is an exact fully faithful functor 
\[
\xymatrix{\Thick_{\cD^+(\Comod_{\Psi}^{I-\mathrm{tors}})}(\cG_{\Psi} \otimes A/I) \ar[r] & \Thick_{\cD^+(\Psi)}(\cG_{\Psi})},
\]
since $A/I \in \Thick_{\cD^+(\Psi)}(\cG_{\Psi})$. Passing to ind-categories thus gives a continuous fully faithful functor 
\[
\xymatrix{\Stable(\Comod_{\Psi}^{I-\mathrm{tors}})  \ar[r] & \Stable_{\Psi}.}
\]
Since $\cG_{\Psi} \otimes A/I \subseteq \Stable(\Comod_{\Psi}^{I-\mathrm{tors}})$, we see that $\Loc(\cG_{\Psi} \otimes A/I) = \Stable_{\Psi}^{I-\mathrm{tors}}$ is contained in the essential image of $\Stable(\Comod_{\Psi}^{I-\mathrm{tors}})$ under the natural inclusion. To prove the converse, let $X \in \Stable(\Comod_{\Psi}^{I-\mathrm{tors}})^{\omega}$ and consider the canonical map $\Gamma^{\Psi}_I X \xr{i} X$. The local cohomology spectral sequence of \Cref{lem:sscomodules} implies that $\epsilon_*(i)$ is an equivalence in $\cD(A)$. By compactness of $X$, it suffices to show that $i$ is an equivalence in $\Stable_{\Psi}$, so that $X \in \Stable_{\Psi}^{I-\mathrm{tors},\omega}$. As both categories are compactly generated and the inclusions are continuous, we obtain that (the essential image of) $\Stable(\Comod_{\Psi}^{I-\mathrm{tors}})$ is also contained in $\Stable_{\Psi}^{I-\mathrm{tors}}$, as desired. 
\end{proof}
We can now prove \Cref{thm:derder}(2) in full generality. 
\begin{proof}[Proof of \Cref{thm:derder}(2).]
	Let $\omega \Stable_{\Psi}^{I-\mathrm{tors}} \subset \cD(\Psi)$ denote the essential image of $\Stable_{\Psi}^{I-\mathrm{tors}}$ under $\omega$ (the functor which inverts quasi-isomorphism). We have shown in \Cref{lem:localomega} that $\omega \Stable_{\Psi}^{I-\mathrm{tors}} \subseteq \cD^{I-\mathrm{tors}}(\Psi)$. For the converse, we recall from \cite[Prop.~5.24]{bhv} that for any $M \in \Stable_{\Psi}$ we have 
	\[
\Gamma_I M \simeq \Gamma_I A \otimes M \simeq \colim_k D(A/I^k) \otimes M
	\]
	where $D(A/I^k) = \iHom_{\Psi}(A/I^k,A)$ denotes the dual of $A/I^k$ in $\Stable_{\Psi}$. Here we can drop the assumption from \cite{bhv} that $A$ is Noetherian by using  Pstr\k{a}gowski's work again, see the proof of \Cref{prop:stableprops}. In particular, since $\omega$ is symmetric monoidal and commutes with colimits, any object in $\omega \Stable_{\Psi}^{I-\mathrm{tors}}$ can be written in the form $\colim \omega D(A/I^k) \otimes \omega M$. By \cite[Lem.~5.13]{bhv} $A/I^k$ is compact and dualizable in $\Stable_{\Psi}$, and so we can apply \cite[Lem.~4.22]{bhv} to see that 
	\[
\omega D(A/I^k) = \omega\iHom_{\Stable_{\Psi}}(A/I^k,A) \simeq \iHom_{\cD(\Psi)}(A/I^k,A).
	\]
	We are thus reduced to showing that the dual $\iHom_{\cD(\Psi)}(A/I^k,A)$ is in $\cD^{I-\mathrm{tors}}(\Psi)$. For brevity, let us denote the dual in $\cD(\Psi)$ by $D_{\Psi}(-)$, so that we are trying to prove that $D_{\Psi}(A/I^k) \in \Loc^{\otimes}(A/I)$. Since $\omega$ is symmetric monoidal, it preserves dualizable objects, so that $A/I^k$ is dualizable in $\cD(\Psi)$, and hence so is $D_{\Psi}(A/I^k)$. By \cite[Lem.~A.2.6]{hps_axiomatic} we see that $D_{\Psi}(A/I^k)$ is a retract of $D_{\Psi}(A/I^k) \otimes D^2_{\Psi}(A/I^k) \otimes D_{\Psi}(A/I^k)$. But $D_{\Psi}^2(A/I^k) \simeq A/I^k$, so that $D_{\Psi}(A/I^k) \in \Loc^{\otimes}(A/I^k)$ as required. It follows that $\omega \Stable_{\Psi}^{I-\mathrm{tors}} \simeq \cD^{I-\mathrm{tors}}(\Psi)$. 

	On the other hand, it follows from the construction of  $\Stable(\Comod_{\Psi}^{I-\mathrm{tors}})$ that after inverting quasi-isomorphisms we recover the usual derived category $\cD(\Comod_{\Psi}^{I-\mathrm{tors}})$ (compare to \cite[Sec.~4.2]{bhv}). 

	It follows that after inverting quasi-isomorphism in \Cref{prop:stabletor} we obtain an equivalence $\cD(\Comod_{\Psi}^{I-\mathrm{tors}}) \simeq \Dtors(\Psi)$ as claimed. 
\end{proof}

\bibliography{duality}\bibliographystyle{alpha}
\end{document}